\documentclass{article}
\usepackage[utf8]{inputenc}

\title{The geometry of the unipotent component of the moduli space of Weil-Deligne representations}
\author{Daniel Funck}
\date{Dec 2024}

\usepackage{graphicx}
\usepackage{amsmath}
\usepackage{amssymb}
\usepackage{amsthm}
\usepackage{amsfonts}
\usepackage{tikz-cd}
\usepackage{mathabx}
\usepackage{tikz}
\usepackage{hyperref}
\usepackage[shortlabels]{enumitem}
\usepackage{dsfont}
\usepackage{mdwlist}

\newtheorem{theorem}{Theorem}[section]
\newtheorem{lemma}[theorem]{Lemma}
\newtheorem{claim}{Claim}
\newtheorem{corollary}[theorem]{Corollary}
\newtheorem{prop}[theorem]{Proposition}

\newtheorem{definition}[theorem]{Definition}
\newtheorem{example}{Example}

\theoremstyle{remark}
\newtheorem*{remark}{Remark}

\theoremstyle{plain} 
\newcommand{\thistheoremname}{}
\newtheorem{genericthm}{\thistheoremname}
\newenvironment{namedtheorem}[1]
  {\renewcommand{\thistheoremname}{#1}%
   \begin{genericthm}}
  {\end{genericthm}}

\DeclareMathAlphabet{\mathpzc}{OT1}{pzc}{m}{it}
\newcommand{\cat}[1]{\mathpzc{#1}}

\newcommand{\leftmapsto}{\reflectbox{$\mapsto$}}
\newcommand{\Gal}{\mathrm{Gal}}

\newcommand{\Hom}{\mathrm{Hom}}

\newcommand{\End}{\mathrm{End}}
\newcommand{\Aut}{\mathrm{Aut}}
\newcommand{\Ind}{\mathrm{Ind}}
\newcommand{\univ}{\mathrm{univ}}
\newcommand{\tr}{\textrm{tr}}
\newcommand{\Nm}{\textrm{Nm}_{F}}
\newcommand{\Spec}{\textrm{Spec}}
\newcommand{\Supp}{\textrm{Supp}}
\newcommand{\Frob}{\textrm{Frob}}
\newcommand{\Stab}{\textrm{Stab}}

\newcommand{\ord}{\textrm{ord}}

\newcommand{\im}{\textrm{im}}
\newcommand{\st}{\textrm{st}}
\newcommand{\Ad}{\textrm{Ad}}
\newcommand{\ad}{\textrm{ad}}
\newcommand{\diag}{\textrm{Diag}}

\newcommand{\depth}{\textrm{depth}}

\newcommand{\Sp}{\textrm{Sp}}
\newcommand{\GSP}{\textrm{GSp}}

\newcommand{\GL}{\textrm{GL}}

\newcommand{\SL}{\textrm{SL}}

\newcommand{\loc}{\textrm{loc}}
\newcommand{\bighotimes}{\widehat{\bigotimes}}
\newcommand{\hotimes}{\widehat{\otimes}}

\newcommand{\Loc}{\textrm{Loc}}

\newcommand{\Lie}{\textrm{Lie}}
\newcommand{\Fil}{\textrm{Fil}}
\newcommand{\gr}{\textrm{gr}}

\newcommand{\id}{\textrm{id}}

\newcommand{\Q}{\mathbb{Q}}
\newcommand{\N}{\mathbb{N}}
\newcommand{\Z}{\mathbb{Z}}
\newcommand{\R}{\mathbb{R}}

\newcommand{\A}{\mathbb{A}^{\infty}_{F^+}}

\newcommand{\F}{\mathbb{F}}
\newcommand{\T}{\mathbb{T}}
\newcommand{\G}{\mathbb{G}}

\newcommand{\Ocal}{\mathcal{O}}

\newcommand{\Gc}{\mathcal{G}}

\newcommand{\Ic}{\mathcal{I}}

\newcommand{\Nc}{\mathcal{N}}
\newcommand{\Oc}{\mathcal{O}}

\newcommand{\Uc}{\mathcal{U}}

\newcommand{\gf}{\mathfrak{g}}

\newcommand{\mf}{\mathfrak{m}}

\begin{document}

\maketitle

\begin{abstract}
    In this paper, we study the moduli space of unipotent Weil-Deligne representations valued in a split reductive group $G$ and characterise which irreducible components are smooth. We apply these smoothness results to show that a certain space of ordinary automorphic forms is a locally generically free module over the corresponding global deformation ring. 
\end{abstract}

\tableofcontents

\section{Introduction and overview}
Let $F$ be a local $p$-adic field and let $G$ be a connected reductive algebraic group over $F$ and $\hat{G}$ its Langlands dual.
The local Langlands conjectures (proven for $GL_n$ by Harris and Taylor in \cite{HT01}) stipulate the existence of a natural map with finite fibres.
\[\frac{\{\textrm{Smooth irreducible representations of } G(F)\}}{\{\textrm{Isomorphism}\}}\rightarrow \frac{\{\textrm{L-parameters of }{}^LG\}}{\{\hat{G}-\textrm{conjugacy}\}}\]

Let $l\neq p$ be a prime. Let $L\subset \bar{\Q}_l$ be an $l$-adic field and $\Ocal$ its ring of integers with residue field $\F$. 
In recent years, through the work found in \cite{BG19}, \cite{Hellmann}, \cite{DHKM20}, \cite{Zhu20} and \cite{FS21}, there has been great interest in studying the properties of a moduli space of L-parameters $\Loc_{\hat{G},\Ocal}$ and a closely related space, the moduli space of framed L-parameters $\Loc_{\hat{G},\Ocal}^{\square}$.

The spaces $\Loc_{\hat{G},\Ocal}^{\square}$ and $\Loc_{\hat{G},\Ocal}$ can be respectively defined as the scheme whose whose $R$-points are the set of L-parameters with $R$ coefficients, and the algebraic stack obtained from the stack quotient of $\Loc_{\hat{G},\Ocal}^{\square}$ modulo the natural action of $\hat{G}$ via conjugation (that is, equivalence of representations):
\[\Loc_{\hat{G},\Ocal}^{\square}(R)=\{\textrm{L-parameters of $\hat{G}$ with $R$-coefficients}\}\]
\[\Loc_{\hat{G},\Ocal}=[\Loc^{\square}_{\hat{G},\Ocal}/\hat{G}]\]
In addition to this interpretation, the scheme $\Loc_{\hat{G},\Ocal}^{\square}$ satisfies the secondary property that the (completions of) local rings of $\Loc_{\hat{G},\Ocal}^{\square}$ can be interpreted as local Galois deformation rings. In this way, it is hoped to better understand these rings, which are crucial ingredients in the Taylor-Wiles-Kisin and Calegari-Geraghty patching methods. 

Let $W_F$ be the Weil group of the field $F$.
In the case of a split reductive group $G$, one would want to define an $L$-parameter as a homomorphism $W_F\rightarrow \hat{G}(R)$ which satisfies some kind of `continuity', but the problem with this naive approach is that the ring $R$ has no topological structure in general. 
Historically, there are multiple solutions to this issue with varying degrees of usefulness. We are interested in the moduli spaces of Bellovin and Gee (\cite{BG19}), defined via Weil-Deligne representations (defined below); and of Dat, Helm, Kurinczuk and Moss (\cite{DHKM20}), defined through representations of a particular dense subgroup $W_F^0\subseteq W_F$ (defined in Section 1.2 of the reference).
\begin{definition}
    A Weil-Deligne representation valued in $G$ with $R$-coefficients is a pair $(r,N)$ where $r:W_F\rightarrow  G(R)$ is a homomorphism with open kernel and $N$ is an element of the nilpotent cone $\Nc_G(R)\subseteq \Lie(G)(R)$ such that $\Ad(g)N=|g|N$ for all $g\in W_F$ where $|.|:W_F\rightarrow F^{\times}\rightarrow \R^{\geq 0}$ is the valuation of $W_F$ coming from local class field theory.
\end{definition}
The two moduli problems of \cite{BG19} and \cite{DHKM20} are representable by algebraic stacks
$\Loc^{BG}_{G,\Ocal}$ and $\Loc_{G,\Ocal}$, respectively, and are disjoint unions of quotient stacks. 

It is known, as in Proposition 2.7 of \cite{DHKM20}, that these definitions give isomorphic moduli spaces over fields of characteristic 0. In positive characteristic $l$ (or mixed characteristic), only the latter moduli space is generally well-behaved, giving the deformation rings as the completions of its local rings.
However, when $l$ is greater than the Coxeter number $h_G$ of $G$,  
the exponential and logarithm maps of Proposition 2.7 of \cite{DHKM20} that arise from Grothendieck's $l$-adic monodromy are well defined polynomials, and we obtain an isomorphism between the two moduli spaces in this case too.

Let $\Ocal$ be a discrete valuation ring (or field) of residue characteristic (resp. characteristic) $l>h_G$ or $0$. Let $\Nc_G\subseteq\mathfrak{g}$ be the nilpotent cone inside the Lie algebra $\mathfrak{g}$.
In this paper, we seek to understand the irreducible components of the scheme studied in \cite{Hellmann}. This is a reduced affine scheme of finite type $S_{G,\Ocal}$, over the ring $\Ocal$,
whose $R$-points ($R$ an $\Ocal$-algebra) are given by
\[S_{G,\Ocal}(R)=\Big\{(\Phi,N) \in G(R) \times \Nc_G(R)| \Ad (\Phi)N =qN\Big\}.\]
where $q\in \Ocal^{\times }$ is some prime power. 
 
The scheme $S_{G,\Ocal}$ is naturally the space of framed unipotent Weil-Deligne representations over $\Ocal$ with values in $G$ (following Definition 2.1.2 of \cite{BG19}). We will, in particular, be interested in the case where $\Ocal$ is the ring of integers in a finite extension of $\Q_l$ because the $\mathfrak{m}_R$-adic completion of the local rings $R$ of the closed points of this scheme can be interpreted as local Galois deformation rings for well behaved $l$. (In fact, whenever the exponential and logarithm maps of Grothendieck's $l$-adic monodromy theorem gives an isomorphism onto a connected component, as above.)
Furthermore, as $S_{G,\Ocal}$ is a connected component of the tame parameters $Z^1(W^0/P,G)_{\Ocal}$, Equation 4.5 of \cite{DHKM20} extends the description of $S_{G,\bar{\Q}_l}$ for various groups $G$ to a description of the geometry of many other connected components of $\Loc_{G,\bar{\Q}_l}^{\square}$ (those which correspond to the case where the action $\Ad_{\varphi}$ is trivial).

In Section 2, we collect some basic results for $S_G$, in the mixed characteristic setting. The results of this paper depend strongly on the technical relationship between the residue characteristic $l$ and the element $q\in\Ocal$. As such, the section begins with defining the notions of $G$-banality and $q$-considerateness, and explaining how they are related. We then give a description of a decomposition of $S_G$ into (unions of) irreducible components, generalising the decomposition of $S_{\GL_n}$ given in Proposition 2.1 of \cite{Hellmann}, as follows:
Let \[p:S_{G,\Ocal}\rightarrow \Nc_G\]
be the projection map onto the second factor. Let $C\subset \Nc_{G,L}$ be a $G$-conjugacy class inside $\Nc_{G,L}$. (We note that, in the case of $\GL_n$, these can be characterised by partitions of $n$ and in this situation we will denote the conjugacy class corresponding to $\lambda$ by $C_{\lambda}$.) We note that because $S_{G,\Ocal}$ is flat over $\Ocal$, the irreducible components biject naturally with those of $S_{G,L}$.
Then $\overline{p^{-1}(C)}\subseteq S_{G,\Ocal}$ is a union of irreducible components of $S_{G,\Ocal}$ (and, in the case of $G=\GL_n$, is itself irreducible).
In Section 3, we expand on and generalise the results of Bellovin \cite{bellovin_2016} Section 7.2 and Proposition 7.10 by proving Theorems \ref{theorem-smooth} and \ref{theorem -notsmooth}, which together state:

\begin{theorem}\label{Theorem: mainthm}
Assume $q$ is considerate towards $G_{\Ocal}$ (see Definition \ref{Def: qconsiderate}).
\begin{enumerate}
    \item Suppose $C\subseteq \Nc_{G,L}$ is a distinguished nilpotent orbit, or the zero orbit with corresponding component $X_C=\overline{p^{-1}(C)}$. Then $X_C$ is a disjoint union of smooth connected components.
    \item Conversely, when $C\subseteq \Nc_{G,L}$ is a non-trivial non-distinguished orbit, the scheme $X_C$ is singular.
\end{enumerate}
\end{theorem}

In Sections 4 and 5 of this paper, we apply the smoothness result of Section 3 to Hida families of ordinary automorphic forms using the Taylor-Wiles-Kisin patching method in a situation very similar to that studied in \cite{geraghty_2018}.
Let $l$ be a prime and $K$ be a finite extension of $\Q_l$ with ring of integers $\Ocal$. Let $F^+$ be a totally real global number field, and consider an imaginary quadratic extension $F$ of $F^+$. 
The Galois representations considered will correspond to certain Hida families of ordinary automorphic forms on a unitary algebraic group $G_D/F^+$
which is a unitary form of a unit group of a division algebra $D/F^+$. 
We will define a certain space of Hida families of ordinary automorphic forms $S^{\ord}(U(l^{\infty}),L/\Ocal)_m$ for $G_D$ with Hecke operators $\T$ and a corresponding deformation ring $R^{univ}_{\mathcal{S}}$. We will then use the Taylor-Wiles patching method to deduce the following theorem.

\begin{theorem}[Theorem \ref{theorem locfree}]
Suppose $l>n$. The module $S^{\ord}(U(l^{\infty}),L/\Ocal)_m^{\vee}[1/l]$ is a finite locally free $R^{univ}_{\mathcal{S}}[1/l]$-module.
\end{theorem}

As a consequence, we can deduce that $R^{univ}_{\mathcal{S}}[1/l]\cong \T[1/l]$, and that the multiplicity of automorphic forms with a given characteristic zero Galois representation is constant along connected components of $R^{univ}_{\mathcal{S}}[1/l]$. In particular, one can extend any such multiplicity results from the classical case to the case of non-classical Hida families. 

\subsection{Acknowledgements}
The research in this work is part of the author's PhD thesis, supported by the Engineering and Physical Sciences Research Council (EPSRC). I would like to thank my supervisor, Jack Shotton, for the advice and support given me throughout this project. I would also like to thank the anonymous reviewer for the detailed comments and suggestions given.

\section{Considerateness and the relation to the stack of L-parameters}

Let $\Ocal$ be a discrete valuation ring or a field with residue field $\F$ of characteristic $l$ or $0$ and fraction field $L$. 
Let $G$ be a connected reductive algebraic group over $\Ocal$, let $\mathfrak{g}$ be its Lie algebra, and let $h_G$ be its Coxeter number. 
Throughout the paper, we assume $l>h_G$ whenever the residue characteristic of $\Ocal$ is finite.

\begin{definition}{\label{Def: qconsiderate}}
    Let $h_G$ be the Coxeter number of $G$. Let $q\in \Ocal^{\times}$ be an element of $\Ocal$ such that $q^k-1$ is invertible in $\Ocal$ for all $k\leq h_G$. When this occurs, we say that $q$ is \textit{considerate} towards $G$ over $\Ocal$.
\end{definition}

In applications, the ring $\Ocal$ will either be a field or the ring of integers in some field extension of $\Q_l$. Notice that, in this case, $q$-considerateness is equivalent to the condition `$1,q,q^2,...,q^{h_G}$ are all distinct in the residue field $\F$' (in a sense, $q$ `treads lightly' around $G$).

\begin{definition}{\label{Def: banal}}
    Let $G$ be a split reductive group over a field $L$ of characteristic $l$. 
    We say:
    \begin{itemize}
        \item $l$ is $G$-banal if $l$ does not divide the order of the finite group $G(\F_q)$.
        \item $l$ is geometrically-$G$-banal if, for any algebraically closed field $E$ of characteristic $l$, any $\phi\in \Loc^{\square}_{G,E}$ can be `Frobenius twisted' by some $g\in C_{G}(\phi(I_F))$ (that is, the centraliser of the inertia subgroup) so that $\phi^g$ is a smooth point of $ \Loc^{\square}_{G,E}$.
        
        The `Frobenius twist' of a representation $\phi:W_F\rightarrow G(L)$ by $g\in C_{G}(\phi(I_F))$ is the representation $\phi^g:W_F\rightarrow G(L)$ which is equal to $\phi$ when restricted to the inertia subgroup, and for which $\phi^g(\Frob)=\phi(\Frob)g$. 
    \end{itemize}
\end{definition}

\begin{remark} Let $\hat{G}$ denote the Langlands dual group of $G$.
    We remark that the notion `$l$ is geometrically-$\hat{G}$-banal' is precisely the notion that `$l$ is ${}^LG$-banal', as defined in Definition 5.27 of \cite{DHKM20}. We introduce the notion simply to remove the extraneous Langlands dual, which is only required in the following Proposition.
\end{remark}

\begin{prop}
    Suppose that $\F$ is a field of positive characteristic $l>h_G$ and that $G$ is a split reductive group. 
Then we have the following implications. 

\begin{itemize}
    \item If $q$ is considerate towards $G_{/\F}$, then $l$ is geometrically-$G$-banal. 
    \item If $l$ is geometrically-$\hat{G}$-banal (that is $^{L}G$-banal), then $l$ is $G$-banal. 
    \item If $G=\GL_n$ or $\SL_n$, then the condition `$l$ is $G$-banal' implies the condition $q$ is considerate towards $\hat{G}_{/\Ocal}$. Thus, all concepts are equivalent.
\end{itemize}
\end{prop}

\begin{proof}
    By definition, $q$ is considerate towards $G_{/\F}$ when the order of $q$ within $\F$ is greater than the Coxeter number $h=h_G$. This is equivalent to $\prod_{n\leq h}\Phi_n(q)\neq0$ inside $\F$ where $\Phi_n$ is the $n$th cyclotomic polynomial. This is the polynomial $\chi^*_{G,1}(q)$ of Theorem 5.7 of \cite{DHKM20} (see Definition B.3). Hence, by Theorems 5.6 and 5.7 of \cite{DHKM20}, it follows that this condition implies that $l$ is $^{L}G$-banal. 

    That $l$ is $^{L}G$-banal implies that $l$ is $G$-banal is a consequence of the Chevalley-Steinberg formula (see. Theorem 25a) of \cite{ChevalleySteinberg}); 
    \[|G(\F_q)|=q^N\prod_{d}(q^d-1)\]
    where $d$ ranges over the fundamental degrees of the Weyl group of $G$. If $l$ divides $\prod_{d}(q^d-1)$, then $l$ certainly divides $\prod_{n\leq h}\Phi_n(q)$ as the Coxeter number is the highest fundamental degree. This shows the second statement by virtue of Theorem 5.7 in \cite{DHKM20}.

    In the case $G=\SL_n$, we get $|\SL_n(\F_q)|=q^N\prod_{i=2}^n(q^i-1)$. Hence, if $l$ is $G$-banal, it follows that the order of $q$ in $\F$ is at least $n$. Thus $q$ is considerate towards $\hat{G}_{/\F}$. The case of $\GL_n$ is similar. 
\end{proof}

\begin{remark}
    It is worth noting that Corollary 5.27 of \cite{DHKM20} gives the criterion that $G$-banal and ${}^LG$-banal are equivalent concepts whenever $G$ is unramified and has no exceptional factors (where triality forms of type $D_4$ are also considered exceptional), but this does not hold in general (see, for example, Remark 5.22 of \cite{DHKM20}). Outside the case of type $A_n$, considerateness is strictly weaker than geometric-$G$-banality. For example, the order of $G=\Sp_6(\F_q)$ is
    \[|\Sp_6(\F_q)|=q^9(q^2-1)(q^4-1)(q^6-1)\]
    and the Coxeter number of $\hat{G}=\textrm{SO}_7$ is $h=6$. However, if $q^5\equiv1 \pmod{l}$, then $l$ is $G$-banal, but $q$ is not considerate towards $\hat{G}_{\Ocal}$.
\end{remark}
We give the following definition:
\begin{definition}
We define the affine scheme $S_{G,\Ocal}$ over $\Ocal$ as the scheme whose $R$ points (for $R$ an $\Ocal$ algebra) are $\big\{(\Phi,N)\in G(R)\times \Nc_G(R): \Ad(\Phi) N =qN\big\}$.
\end{definition}
Corollary 5.4 of \cite{bellovin_2016} shows that this is a reduced scheme when $\Ocal$ is a characteristic zero field, and hence, is a variety (i.e. a finite-type, separated, reduced scheme over a field). 
As discussed in the introduction, we may picture $S_{G,\Ocal}$ as the moduli space of unipotent Weil-Deligne representations $(r,N)$ valued in $G_{\Ocal}$. In this context, the adjective `unipotent' means that the restriction of $r$ to the inertia subgroup $I_F$ is trivial (that is, $r(I_F)=1$).

\begin{prop}\label{prop2.4}

\begin{enumerate}
    \item Suppose $q$ is considerate towards $G_{/\Ocal}$. If $(\Phi,N)\in G\times \gf$ satisfies $\Ad(\Phi)N=qN$, then $N\in\Nc_G$. Hence, 
    \[S_G(R)=\big\{(\Phi,N)\in G(R)\times \gf(R): \Ad(\Phi) N =qN\big\},\]
    and the requirement that $N\in\Nc_G(R)$ is redundant.  
    \item When $G$ is split and $l>h_G$, then $S_{G,\Ocal}$ is isomorphic to a closed subscheme of the moduli space of tame parameters $Z^1(W^0_F/P_F,G)_{\Ocal}$ (see Section 1.2 of \cite{DHKM20} for a definition of this space).
    \item In addition, when $l$ is geometrically-$G$-banal, this space is a connected component of $Z^1(W^0_F/P_F,G)_{\Ocal}$.
\end{enumerate}
\end{prop}

\begin{proof}
Because $l> h_G$, the prime $l$ is very good in the notation of Section 2.4 of
    \cite{Cotner}. Hence, by Theorem 4.13 of \cite{Cotner}, we have an isomorphism of $\Ocal$-algebras
     \[\Ocal[\mathfrak{g}]^G\rightarrow \Ocal[\mathfrak{t}]^W\]
     given by the restriction of functions on $\gf$ to $\mathfrak{t}$, where $\mathfrak{t}$ is a Cartan subalgebra of $\gf$ and $W$ is the Weyl group. 
    
    By chapter 3 of \cite{hump} (see table 1 of Section 3.7, and table 2 of Section 3.18) the generators of $\Ocal[\mathfrak{t}]^W$ are homogeneous of degree at most the Coxeter number $h_G$, and hence the same is true for $\Ocal[\mathfrak{g}]^G$. We note that while this reference restricts to the case of a field of characteristic zero, the results extend to $\Ocal$ because $|W|$ is invertible inside $\Ocal$ and $\Ocal[\mathfrak{t}]^W$ is a free $\Ocal$-module. 
    
    Let $s$ be a generator of $\Ocal[\mathfrak{g}]^G$ and suppose $(\Phi,N)\in G(R)\times \gf(R)$ satisfies $\Ad(\Phi)N=qN$. Then as $s$ is $G$-invariant and homogeneous of degree at most the Coxeter number $h_G$, 
    the condition
    $s(\Ad(\Phi)N)=s(qN)$ implies $s(N)=q^is(N)$ for some $i\leq h_G$. As $q$ is considerate towards $G_{/\Ocal}$, we see that $q^i-1$ is a non-zero divisor in $\Ocal$, and hence that $s(N)=0$. 
    Thus the image of $N$ in the GIT quotient $\mathfrak{g}//G$ is zero. Since $l$ is very good, Theorem 4.12 of \cite{Cotner} shows that $N$ lies in the set of $R$-points of the nilpotent cone. Part 1 of the proposition follows.

    Suppose $G$ is a split group. As $l\neq p$, the space $Z^1=Z^1(W^0_F/P_F,G)_{\Ocal}$ has a model as a flat affine scheme over $\Ocal$ with $R$-points equal to 
    \[Z^1(W^0_F/P_F,G)_{\Ocal}(R)=\big\{(\phi,\sigma)\in G(R)^2: \phi\sigma\phi^{-1}=\sigma^q\big\}.\]
    Since $l>h_G$, the exponential and logarithm maps of Section 6 of \cite{Balaji_2017} are well defined polynomials, and thus we have an isomorphism between the nilpotent cone in $\Nc_G$ and unipotent cone $\mathcal{U}_G$. Hence, we have a map 
    \begin{align*}
          S_{G,\Ocal} &\rightarrow Z^1(W^0_F/P_F,G)_{\Ocal} \\
          (\Phi,N) & \mapsto (\Phi,\exp{N})
    \end{align*}
    which is an isomorphism onto the closed subscheme of $Z^1(W^0_F/P_F,G)_{\Ocal}$ given by those elements $(\phi,\sigma)$ with $\sigma\in \mathcal{U}\subseteq G$, where $\mathcal{U}$ is the unipotent cone. 
    
    For part 3, suppose $l$ is geometrically-$G$-banal. Let $\mathcal{U}^+$ be the scheme-theoretic image of $Z^1(W^0_F/P_F,G)_{\Ocal}$ through the second projection onto $G$.
Proposition 2.6 of \cite{DHKM20} tells us the underlying reduced scheme of $\Uc^+$ is a subscheme of $\{\sigma\in G//G: \sigma^M=1 \}$ for some fixed $M\in\N$. Thus $\Uc^+$ is $0$-dimensional over $\Ocal$. As $Z^1(W^0_F/P_F,G)$ is flat and $\Ocal$ is a discrete valuation ring, we conclude $\mathcal{U}^+$ is a finite flat $\Ocal$-scheme. 

We claim that no two distinct $\Ocal$-points of $\mathcal{U}^+$ reduce to the same $\F$-point. The preimages of the $L$ (resp. $\F$) points of $\Uc^+$ in $Z^1(W_F^0/P_F,G)_L$ (resp. $Z^1(W_F^0/P_F,G)_{\F}$) are (unions of) connected components, thus it suffices to show that no two distinct points in $\Uc^+(L)$ reduce to the same point in $\Uc^+(\F)$. This follows in turn, from the statement that no two connected components of $Z^1(W_F^0/P_F,G)_L$ reduce to the same component of $Z^1(W_F^0/P_F,G)_{\F}$, which follows from Proposition 5.26 of \cite{DHKM20}, which states that $Z^1(W^0_F/P_F,G)_{\F}$ is reduced.
We can conclude that the point $\Spec\Ocal \hookrightarrow \Uc^+$ defined by $\sigma=1$ is a connected component of $\Uc^+$. It follows that $S_G$, the preimage of this point, is a connected component of $Z^1(W^0_F/P_F,G)_{\Ocal}$.

\end{proof}

We will also need the following results.

\begin{prop}\label{prop 9}
\begin{enumerate}
    \item The algebraic group $G$ acts on $S_G$ via simultaneous conjugation
$$g.(\Phi,N)=(g\Phi g^{-1},\Ad(g)N).$$
    \suspend{enumerate}
For parts 2-5, assume that $l$ is geometrically-$G$-banal.
    \resume{enumerate}
    \item The scheme $S_{G,\Ocal}$ is flat and a local complete intersection of relative dimension $\dim G$ over $\Ocal$.
   
    \item Define the second projection map $p:S_G\rightarrow \Nc_G$ as before. If $C$ is a $G_{/L}$-conjugacy class inside $\Nc_{G,L}\subseteq \Nc_G$, then the closed subscheme $X_C:=\overline{p^{-1}(C)}\subset S_G$ is a union of irreducible components and $S_G=\bigcup_C X_C$. 
    \item If in addition $G=\GL_n$, the $X_C$ are irreducible components of $S_{n,\Ocal}:=S_{\GL_n,\Ocal}$ and these can be naturally identified with partitions of $n$. For a partition $p$, we call the corresponding component $X_p$.
     \item The scheme $S_{G,\Ocal}$ is reduced.
\end{enumerate}
\end{prop}
\begin{proof}
\begin{enumerate}
    \item This is clear.
    \item This follows from Proposition \ref{prop2.4}(3) and Corollary 2.5 of \cite{DHKM20}.
    \item As $S_{G,\Ocal}$ is flat over $\Ocal$, the irreducible components of $S_{G,\Ocal}$ are exactly those of the open subscheme $S_{G,L}$. This then follows from the proof of part 2, after noticing that $\Nc_{G,L}=\bigcup_C C_L$ as sets.
    \item Suppose $G=\GL_n$. Then $C$ is a quotient of $\GL_n$, and so is irreducible. Because centralisers inside $\GL_n$ are irreducible, the map $p^{-1}(C)\rightarrow C$ is flat with irreducible smooth fibres, and thus is smooth and open. 
    By \cite[\href{https://stacks.math.columbia.edu/tag/004Z}{Lemma 004Z}]{stacks-project}, it follows that $p^{-1}(C)$ is irreducible, and thus so is $X_C$. The final claim follows from the theory of Jordan normal forms. 
\item This follows from Proposition 2.8 of \cite{DHKM20} and Proposition \ref{prop2.4}(3).
\end{enumerate}
\end{proof}

\subsection{Lemmas in commutative algebra and algebraic geometry}

The remaining part of this section proves some lemmas from algebraic geometry and commutative algebra that we will need later.

\begin{lemma}
Let $G$ be a smooth algebraic group over a scheme $S$, and let $X$ be an $S$-scheme. Suppose that we have a morphism $m:G\times_S X\rightarrow X$ defining a group action of $G$ on $X$. Then $m$ is a smooth morphism. 
\end{lemma}
\begin{proof}
First, the morphism $p_X:G\times_S X\rightarrow X$ obtained by the base change of $G\rightarrow S$ is a smooth morphism. 
Next, consider the automorphism $\phi$ of $G\times_S X$ given by $(g,x)\mapsto (g,g.x)$. As this is an isomorphism, it is smooth. 
Then, because $m=p_X\circ \phi$ is a composition of smooth morphisms, it is also smooth. 
\end{proof}

\begin{lemma}\label{lemmasmooth}
Let $f:X\rightarrow Y$ be a smooth morphism of schemes. Let $p\in X$. Then $Y$ is regular at $f(p)$ if and only if $X$ is regular at $p$. 
\end{lemma}

\begin{proof}
After reducing the problem to local ring maps on stalks, this follows from Theorem 23.7 of \cite{Matsumura}.
\end{proof}

\begin{lemma}\label{lemma-completion} Assume that $\Ocal$ is complete. 
Let $R$ be a local $\mathcal{O}$-algebra, and assume that it is topologically of finite type with respect to the $\mathfrak{m}_R$-adic topology. Let $x$ and $\bar{x}$ be prime ideals of $R$ that give rise to the following commutative diagram. 

\begin{tikzcd}
& & R \ar[r, "x"] \ar[rd,"\bar{x}"] & \mathcal{O} \ar[d,two heads] \ar[r, hook]& L=\mathcal{O}[\frac{1}{l}]\\
& & & \mathbb{F}
\end{tikzcd}\newline
Then $$R_{\bar{x}}^{\wedge}\left[\frac{1}{l}\right]_x^{\wedge}\cong R_x^{\wedge}$$
\end{lemma}
\begin{proof}
Notice that since $R\setminus x \supseteq R\setminus\bar{x}\cup{\{\frac{1}{l}\}}$, that $R_{\bar{x}}\left[\frac{1}{l}\right]_x\cong R_x$.
As $R$ is Noetherian, we have $\bigcap_{n}\bar{x}^n=0$, and thus we have an injection $R_{\bar{x}}\hookrightarrow R_{\bar{x}}^{\wedge}$. This gives us a local homomorphism inclusion 
$$R_x=R_{\bar{x}}\left[\frac{1}{l}\right]_x\hookrightarrow R_{\bar{x}}^{\wedge}\left[\frac{1}{l}\right]_x.$$
We notice that $R_x/x\cong L$, so that
\[\left[R_{\bar{x}}^{\wedge}\left[\frac{1}{l}\right]_x\right]/x\cong \left[\varprojlim_n (R/\bar{x}^n)/x\right][1/l]\cong \varprojlim_n \big(R/(x,l^n)\big)[1/l]\cong (\varprojlim\Ocal/l^n)[1/l]=L,\]
the last equality arising because $\Ocal$ is complete.
Thus, by \cite[\href{https://stacks.math.columbia.edu/tag/0394}{Lemma 0394}]{stacks-project}, we have that 
$R_{\bar{x}}^{\wedge}\left[\frac{1}{l}\right]_x^{\wedge}$
 is the completion of $R_{\bar{x}}^{\wedge}\left[\frac{1}{l}\right]_x$ under the $x$-adic topology arising from $R_x$, and is a finite $R_x^{\wedge}$-module. 
 It follows that the map
$$R_x^{\wedge}\rightarrow R_{\bar{x}}^{\wedge}\left[\frac{1}{l}\right]_x^{\wedge}$$
is an injection, and induces an isomorphism modulo $x$. If $C$ is the cokernel (which we now know to be a finite $R_x^{\wedge}$-module) then we see $xC=C$, and so Nakayama's Lemma shows us that $C=0$, implying that the map is an isomorphism. 
\end{proof}

\section{Smoothness results for $X_C$}

The goal of this section is to prove Theorem \ref{Theorem: mainthm}. Let $G$ be a connected reductive group over $\Ocal$ and $S_{G,\Ocal}$ as before. As each map $X_C\rightarrow \Spec(\Ocal)$ is flat, we can (and do) reduce the problem to the case $\Ocal=\F$ is a field of characteristic $0$ or $l$ (see \cite[\href{https://stacks.math.columbia.edu/tag/01V8}{Lemma 01V8}]{stacks-project}). Since smoothness is an fpqc-local property, we can assume that $\F$ is algebraically closed. We make these assumptions throughout this section.

\subsection{Associated cocharacters}

In what follows, we will require a little bit of set-up, notation, and knowledge of Bala-Carter theory. Let $G$ be a connected reductive group over an algebraically closed field $\F$ with Lie algebra $\gf$ and let $C\subseteq \Nc_G$ be a nilpotent orbit.
In what follows, we restrict to the case where the derived subgroup of $G$ is (almost) simple. When $G$ is of adjoint type, we can do this because then $G=\prod_i  G_i$ for $G_i$ almost simple, adjoint, and $S_{G.\Ocal}\cong \prod_i{S_{G_i,\Ocal}}$. If $G$ is not of adjoint type, then $S_{G,\Ocal}\rightarrow S_{G^{ad},\Ocal}$ is a $Z(G)$-torsor, and since $Z(G)$ is smooth (under our considerateness condition), any smoothness result translates between the cases for $G$ and $G^{ad}$. 


Let $\tilde{T}$ be a maximal torus of $G$ and $\Pi$ the set of roots. Let $e\in C\subseteq \gf$.
Let $L\subseteq G$ be a Levi subgroup of $G$, minimal subject to $e\in\Lie(L)$. Let $Z_L$ be the centre of $L$. 
Following Definition 2.8 of \cite{FowRoh08} (see also Section 2.3 of \cite{Premet03}), the nilpotent element $e$ is called distinguished in $L$ if $Z_L^{\circ}$ is a maximal torus of the centraliser $C_L(e)$ of $e$ in $L$.
By Proposition 2.11 of \cite{FowRoh08}, there is an associated cocharacter $\lambda:\mathbb{G}_m\rightarrow \tilde{T}$ such that $\Ad(\lambda(t)).e=t^2e$ and $\im(\lambda)\subseteq [L,L]$.

The group $\G_m$ acts on $\Lie(G)$ through $\Ad\circ\lambda:\G_m\rightarrow \Aut(\gf)$, which gives a decomposition
\[\gf=\bigoplus_{i\in\Z}\gf(\lambda,i).\]
Through Lemma 5.6.5 of \cite{Carter85(93)} and the preceding discussion, we choose a base of simple roots $\Delta\subseteq \Pi$ such that $\langle\alpha,\lambda\rangle\geq 0$ for all $\alpha\in\Delta$. Call the corresponding Borel subgroup $B$.

We define a parabolic subgroup $P_{\lambda}\subseteq L$ such that $\Lie(P_{\lambda})=\bigoplus_{i\geq 0}\gf_L(\lambda,i)$ where each $\gf_L(\lambda,i)=\gf(\lambda,i)\cap \Lie(L)$. 
We note that $P_{\lambda}$ has a Levi decomposition $P_{\lambda}=M_{\lambda}U_{\lambda}$ with $\Lie(U_{\lambda})=\bigoplus_{i>0}\gf_L(\lambda,i)$ and $\Lie(M_{\lambda})=\gf_L(\lambda,0)$. We say $P_{\lambda}$ is a distinguished parabolic subgroup in $L$ if $\dim\gf_L(\lambda,0)=\dim(\gf_L(\lambda,2))+\dim(Z_L)$, and it is a result of \cite{Premet03} (Proposition 2.5) that $e$ is distinguished if and only if $P_{\lambda}$ is distinguished. 

The primary result of Bala-Carter theory is that there is a bijection between the adjoint orbits of $\Nc_G$ and pairs $(M,P)$ where $M$ is a Levi subgroup of $G$, and $P$ is a distinguished parabolic subgroup of $M$.

\subsection{The smoothness result for irreducible components corresponding to distinguished nilpotent orbits}

This subsection contains the main result of the paper.

\begin{lemma}\label{Lem:unipotent}
Assume $e\in C$ is a distinguished nilpotent element with $\lambda$ an associated cocharacter. Let $C_G(e)$ be the centraliser of $e$ in $G$ with Levi decomposition $C_G(e)=MR$ where $R=R_u(C_G(e))$ is unipotent and $M$ is reductive. Suppose that $t\in\G_m$ is sufficiently generic so that $\gf^{\Ad(\lambda(t))}=\gf(\lambda,0)$. Then every element of $C_G(e)\lambda(t)$ is conjugate to an element of $M\lambda(t)$.
\end{lemma} 

\begin{proof}

Since $e$ is distinguished, Theorem A of \cite{Premet03} tells us that $M=M_{\lambda}\cap C_G(e)$ and $R=U_{\lambda}\cap C_G(e)$. Further, following Definition 2.8 of \cite{FowRoh08}, the maximal torus of $M$ is $Z_G^{\circ}$ so that $M/Z_G^{\circ}$ is a rank $0$ reductive group; ergo finite. Thus, all unipotent elements of $C_G(e)$ lie in $R$. 
Let $g\in\lambda(t)C_G(e)\subseteq \tilde{T}C_G(e)$ have abstract Jordan decomposition $g=su$ with $s$ semisimple and $u$ unipotent. As $s$ is semisimple, it lies in a maximal torus $T'$ of $\tilde{T}C_G(e)$ and so there is some $x\in \tilde{T}C_G(e)$ such that $xsx^{-1}\in \tilde{T}$. In fact, we can assume without loss of generality that $x\in C_G(e)$ because $\tilde{T}$ is abelian. Then, because $\lambda(t)$ normalises $C_G(e)$, we see that $xgx^{-1}\in \lambda(t)C_G(e)$. Since $u$ is unipotent in $\tilde{T}C_G(e)$ we obtain $u\in C_G(e)$ and thus, we get 
\[xsx^{-1}=xgx^{-1}(xux^{-1})^{-1}\in\lambda(t)C_G(e)\cap \tilde{T}.\]
Hence 
\[xsx^{-1}\in \tilde{T}\cap \lambda(t)C_G(e)=\tilde{T}\cap \tilde{T}M\cap \lambda(t) C_G(e)=\lambda(t)\tilde{T}\cap \lambda(t)M=\lambda(t)[\tilde{T}\cap M].\]
Then $xux^{-1}\in C_G(xsx^{-1})=C_G(\lambda(t)s')$ for some $s'\in \tilde{T}\cap M$. By the genericity condition on $t$, we see that $C_G(\lambda(t))=M_{\lambda}$ and so (by Theorem 3.5.3 of \cite{Carter85(93)}) $C_G(\lambda(t)s')\subseteq C_G(\lambda(t))=M_{\lambda}$. Hence $xux^{-1}\in M$ and $xux^{-1}=1$. The result follows. 
\end{proof}
We are now equipped to prove the smoothness result:
\begin{theorem}\label{theorem-smooth}
    Let $G_{/\Ocal}$ be a connected reductive group with centre $Z$, let $\gf$ be the Lie algebra of $G$, and suppose $q\in \Ocal$ is considerate towards $G$ over $\Ocal$. Suppose $C\subseteq \Nc_{G,L}$ is either $0$ or a distinguished nilpotent adjoint orbit. Then $X_C$ is smooth over $\Ocal$ and there is a bijection between the connected components of $X_C$ and the set of $\Phi_0$-twisted conjugacy classes of the group $\pi_0(C_G(e))$.
\end{theorem}

\begin{proof}
    Consider first the case $C=0$. Then $X_C=\big\{(\Phi,0)\in S_{G,\Ocal}\big\}\cong G$ via the map projecting to the $\Phi$ coordinate. Since $G$ is smooth, this proves the theorem. 

Keep the notation of before, with $e$ a distinguished nilpotent element of $\gf$ with associated cocharacter $\lambda:\G_m\rightarrow T$. After making some choice of $\sqrt{q}\in \F$, set $\Phi_0=\lambda(\sqrt{q})$. 
    The centraliser $C_G(e)$ exhibits a Levi decomposition $C_G(e)=MR$ with $M$ reductive and $R$ the unipotent radical. 
    Because $e$ is a distinguished nilpotent element in $\gf$, we see that $Z^{\circ}$, the connected component of the identity of the centre $Z$ of $G$, is a maximal torus of $M$ by Proposition 2.11 (iii) of \cite{FowRoh08}. Since $M/Z^{\circ}$ is a split reductive group of rank $0$, it is finite; because unipotent radicals are connected, it is 
    isomorphic to the component group $A(e)$ of the centraliser $C_G(e)$. 
We choose a set of representatives $S$ such that $M=\coprod_{s\in S} sZ^{\circ}$.

    Define $Y=M\Phi\times \gf(\lambda,2)$. Since the characteristic $l>h$, we see that $M$ is a smooth subgroup of $G$ making $Y$ a smooth $\F$-scheme. It is clear that $Y$ is a closed subscheme of $S_G$ because if $(m\Phi_0,N)\in M\Phi_0 \times \gf(\lambda,2)$ 
then
\begin{align*}
    \Ad(m\Phi_0).N 
    &=\Ad(m)\Ad(\Phi_0).N \\
    &=\Ad(m).qN \\
    &=qN.
\end{align*}
In fact, $Y$ lies inside the irreducible component $X_C$. The distinguished element $e$ lies in the unique open dense $P_{\lambda}$-orbit inside $\gf(2,\lambda)$ (See, for example, Proposition 5.8.7 b) of \cite{Carter85(93)}) and thus $P_{\lambda}.\left[M\Phi_0\times \{e\}\right]$ is a dense open subscheme of $Y$. As $M\Phi_0\times \{e\}\subseteq p^{-1}(C)$, we obtain $P_{\lambda}.\left[M\Phi_0\times \{e\}\right]\subseteq p^{-1}(C)$ and consequently $Y\subseteq X_C$. 
    Define the morphism
\begin{align*}
    f:G\times Y &\rightarrow X_C \\
    (g,(\Phi,N)) &\mapsto (g\Phi g^{-1},\Ad(g)N).
\end{align*}
As $G\times Y$ is a smooth variety, Lemma \ref{lemmasmooth} implies the theorem provided we can show that $f$ is a smooth surjective morphism. 

We note that surjectivity is equivalent to the statement `every pair $(\Phi,N)\in X_C$ is conjugate to a pair in $Y$'. To prove this, it is sufficient to show that $\Phi$ is conjugate to an element of $M\Phi_0$ whenever $(\Phi,N)\in X_C$.
As there is some $\Ad(g).e\in \gf$, upon which $\Phi$ acts as a multiplication by $q$, it is clear that $\Phi$ is conjugate to some element of $C_G(e)\Phi_0$. So, by Lemma \ref{Lem:unipotent}, we see that $\Phi$ is conjugate to an element of $M\Phi_0$. This proves surjectivity.

We now proceed to prove that $f$ is smooth. Consider the following commutative diagram
\[
\begin{tikzcd}
G\times Y\arrow{r}{f} \arrow[d] & X_C \arrow[d] \\
G\times M\Phi_0 \arrow{r} & G
\end{tikzcd}
\]
where the vertical maps come from the ``forget $N$" projections $(g,m\Phi_0,N)\in G\times Y\mapsto (g,m\Phi_0)\in G\times M\Phi_0$ and $(\Phi,N)\in X_C\mapsto \Phi\in G$ respectively, and the horizontal maps are defined via the conjugation action of $g\in G$ on $Y$ so that the diagram commutes.
Choosing a set of representatives $S$ of $M/Z^{\circ}$, the map $G\times M\Phi_0\rightarrow G$ factors through
\[\begin{tikzcd}
G\times M\Phi_0\arrow{r} \arrow[d, equal] & G \\
\coprod_{s\in S} G\times sZ^{\circ}\Phi_0 \arrow{r}{m} & \coprod_s Z^{\circ}G_{s\Phi_0} \arrow[u]
\end{tikzcd}\]
where $G_{s\Phi_0}$ denotes the conjugacy class of $s\Phi_0$ in $G$.
Note that each $Z^{\circ}G_{s\Phi_0}$ defines a locally closed subvariety of $G$. If any two subschemes $Z^{\circ}G_{s\Phi_0}$ and $Z^{\circ}G_{t\Phi_0}$ intersect in $G$, say $x\in Z^{\circ}G_{s\Phi_0}\cap Z^{\circ}G_{t\Phi_0}$, then 
\[z_1g_1s\Phi_0g_1^{-1}=x=z_2g_2t\Phi_0g_2^{-1}\]
implies that $t\Phi_0=(z_1z_2^{-1})(g_2^{-1}g_1) s\Phi_0 (g_2^{-1}g_1)^{-1}$; and thus leads us to $Z^{\circ}G_{s\Phi_0}=Z^{\circ}G_{t\Phi_0}$ as locally closed subschemes of $G$. Thus, by possibly restricting to a subset $S'$ of the set of representatives $S$ if necessary, we can view $\coprod_{s\in S'} Z^{\circ}G_{s\Phi_0}$ as a locally closed subscheme of $G$.

Because the map $f:G\times Y\rightarrow X_C$ is surjective, the map $X_C \rightarrow G$ also factors through $\coprod_{s\in S'} Z^{\circ}G_{s\Phi_0}$, giving us the commutative diagram:
\begin{equation}\label{commutative diagram}
\begin{tikzcd}
G\times Y\arrow{r}{f} \arrow[d] & X_C \arrow[d] \\
G\times M\Phi_0 \arrow{r}{m} & \coprod_{s\in S'} Z^{\circ}G_{s\Phi_0}
\end{tikzcd}
\end{equation}

We claim that this is a pullback square. 
Since $e$ is distinguished, we see
$\gf(\lambda,1)=0$ and thus each simple root $\alpha$ has its corresponding character eigenspace $\gf_{\alpha}$ either contained inside $\gf(\lambda, 0)$ or $\gf(\lambda, 2)$. 
Hence, as every positive root is the sum of at most $h-1$ simple roots (where $h=h_G$ is the Coxeter number of $G$), we see that 
    \[\{i\in\Z:\gf(\lambda,i)\neq 0\}\subseteq 2\Z\cap [-2h+2,2h-2].\]
Thus, given that $q$ is considerate towards $G_{/\F}$, it follows that the subspace of $\gf$ upon which $\Phi_0=\lambda(\sqrt{q})$ acts as multiplication by $q$ is precisely $\gf(\lambda,2)$. 

Thus, if $(g,\Phi)\in G\times M\Phi_0$ and $(\Phi',N')\in X_C$ such that $g\Phi g^{-1}=\Phi'$, then $\Ad(\Phi)(\Ad(g^{-1})N')=q\Ad(g^{-1})N'$ 
and $\Ad(g^{-1})N'\in \gf(\lambda,2)$ by the previous discussion. So the morphism 
\[((g,\Phi), (\Phi',N'))\rightarrow (g,(\Phi, \Ad(g^{-1})N'))\]
gives an inverse to the natural morphism $G\times Y\rightarrow (G\times M\Phi_0)\times_{\coprod_s Z^{\circ}G_{s\Phi_0}}X_C$. This shows the above commutative diagram is a pullback square.

Now, by the theory of homogeneous spaces, the bottom map $m$ is flat with fibres isomorphic to $\Stab_G(\Phi_0)$, which are smooth group schemes. This shows that $m$ is smooth. Hence, since the map $f$ is the base change of $m$ to $X_C$, by Proposition 10.1 of \cite{Hartshorne} we see that $f$ is smooth. We can thus conclude $X_C$ is smooth over $\F$.

The statement on the number of connected components is Theorem 2.5 of \cite{Jackcomponents}, and is included for completeness.
\end{proof}

\begin{remark}
    A question arises regarding the generality and necessity of the considerateness condition. That is, when exactly is $q$-considerateness a necessary condition for smoothness? As one can see in the proof, we used considerateness to prove that $G\times Y$ was the pullback of the diagram
    \[
\begin{tikzcd}
G\times Y\arrow[r]{f} \arrow[d] & X_C \arrow[d] \\
G\times M\Phi_0 \arrow[r]{m} & M.G_{\Phi_0}
\end{tikzcd}
\]
  arising from the fact that $\{N\in \gf: \Ad(\Phi_0)N=qN\}=\gf(\lambda,2)$. When $C$ is the \emph{regular} nilpotent orbit, then $\gf(\lambda,i)\neq 0$ for every $i\in [-2h+2,2h-2]\cap 2\Z$, so we see that $q$-considerateness is precisely the condition that $\{N\in \gf: \Ad(\Phi_0)N=qN\}=\gf(\lambda,2)$. When $C$ is distinguished and non-regular, we have $\gf(\lambda,2h-2)=0$. Hence, there is some $r<h$, depending on the distinguished orbit $C$, such that whenever $q\in \Ocal$ has the property that $1,q,...,q^r$ are distinct, then $X_C$ is smooth via the above proof. This $r$ can always be taken to be 
  \[r=1+\max\{i:\gf(\lambda,2i)\neq 0\}.\]
\end{remark}

\begin{remark}
    Let $X_C^s$ be the image of $f$ restricted to $G\times sZ^{\circ}\Phi_0$ (so that $X_C^s$ is an irreducible component of $X_C$).
    Another way of interpreting the Cartesian diagram \ref{commutative diagram} is that $X^s_C\rightarrow Z^\circ G_{s\Phi_0}$ is the total space of the vector bundle on $Z^\circ G_{s\Phi_0}$ whose fibre above $\Phi$ is $\ker(\Ad(\Phi)-q)\subseteq \gf$. In particular, we obtain the corollary:
    
\end{remark}

\begin{corollary}
    If $C$ is a distinguished orbit, and $s,\Phi_0$ are as in the proof of Theorem \ref{theorem-smooth} then $X^s_C$ is described as the closed subscheme of $G\times \gf$ cut out by the equations:
    \begin{itemize}
        \item $\Ad(\Phi)N=qN$;
        \item any set of equations describing the closed orbit $Z^{\circ}G_{s\Phi_0}\subseteq G$.
    \end{itemize}
\end{corollary}

\subsection{The converse non-smoothness result}

We now begin work towards the converse of Theorem $\ref{theorem-smooth}$. Consider the situation described at the beginning of the chapter with reductive group $G$ over an algebraically closed field $\F$ with (almost) simple derived subgroup with maximal torus $T$, set of roots $\Pi$, a set of simple roots $\Delta\subseteq \Pi$, and a nilpotent element $e$ with associated cocharacter $\lambda$. 
Let $L$ be the smallest Levi subgroup of $G$ with $e\in\Lie(L)$, so that $e$ is distinguished for $L$. Let $\Delta_L\subseteq \Delta$ be the simple roots of $L$. 

\begin{theorem}{\label{theorem -notsmooth}}
Let $G$ be as before. Suppose $C\subseteq \Nc_G$ is a nilpotent adjoint orbit, distinguished in a proper and non-trivial Levi subgroup $T\neq L\subsetneq G$. Then $X_C\subseteq S_G$ is singular. 
\end{theorem}
Before we move onto the proof of this theorem, we will need some terminology and a lemma. 
Let $D=(\Delta,\Sigma)$ be the Dynkin diagram of $G$, and $D_L=(\Delta_L,\Sigma_L)\subseteq D$ the maximal subdiagram containing exactly the vertices $\Delta_L$. (Note that $D$ is connected when its derived subgroup is (almost) simple but $D_L$ may not be connected). Recall that, given a distinguished nilpotent $e$, one can attach to each vertex $\beta\in \Delta_L$ the number $\langle\alpha,\lambda\rangle\in\{0,2\}$, and this is called the weighted Dynkin diagram $D_L(e)$.

We call a root $\alpha\in\Delta_L$ \emph{exposed} if there is an edge in $D$ connecting $\alpha$ to a root $\beta\in\Delta\setminus\Delta_L$.

\begin{lemma}{\label{Lem:exposed roots}}
 Any exposed root $\alpha\in \Delta_L$ has $\langle\alpha,\lambda\rangle=2$.
\end{lemma}
\begin{proof}
    In the case of type $A,B,C$ and $D$, either all simple factors of the Levi subgroup is of type $A$, or exactly one of the almost simple factors of $L$ is of type $B,C,D$ respectively. 
    If all the simple factors of the Levi subgroup are type $A$, then all roots $\alpha\in\Delta_L$ have $\langle \alpha,\lambda\rangle =2$ because all distinguished nilpotents orbits are regular in type $A$. In the case with one factor of type $B,C$ or $D$, the only `exposed' root of this factor is on the end of the string, and one can see from the tables on pages 174 and 175 of \cite{Carter85(93)} that, independent of the choice of distinguished nilpotent, this exposed root $\alpha$ always has $\langle\alpha, \lambda \rangle =2$. This proves the lemma in types $A,B,C$ and $D$. 

    In type $G_2$, the only proper Levi subgroups are of type $A_1$, so all roots $\alpha\in \Delta_M$ have $\langle\alpha, \lambda \rangle =2$.
    
    In type $F_4$, there are three possibilities for a Levi factor not of type $A$, these being $C_2,C_3$ and $B_3$. The distinguished orbits of these Levi subgroups are described on pages 174 and 175 of \cite{Carter85(93)} and are listed below in Figure \ref{fig:Levis of F_4}.
    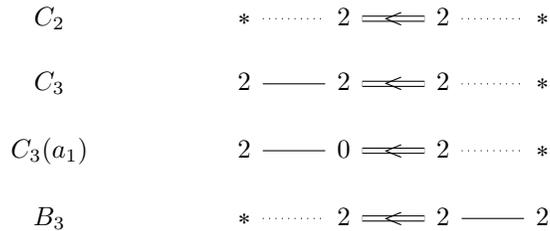
\begin{figure}[ht]
        \centering
\begin{tikzcd}[row sep=small]
	{C_2} && {*} & 2 & 2 & {*} \\
	{C_3} && 2 & 2 & 2 & {*} \\
	{C_3(a_1)} && 2 & 0 & 2 & {*} \\
	{B_3} && {*} & 2 & 2 & 2
	\arrow[dotted, no head, from=1-3, to=1-4]
	\arrow["{<}"{marking, allow upside down}, Rightarrow, no head, from=1-4, to=1-5]
	\arrow[dotted, no head, from=1-5, to=1-6]
	\arrow[no head, from=2-3, to=2-4]
	\arrow["{<}"{marking, allow upside down}, Rightarrow, no head, from=2-4, to=2-5]
	\arrow[dotted, no head, from=2-5, to=2-6]
	\arrow[no head, from=3-3, to=3-4]
	\arrow["{<}"{marking, allow upside down}, Rightarrow, no head, from=3-4, to=3-5]
	\arrow[dotted, no head, from=3-5, to=3-6]
	\arrow[dotted, no head, from=4-3, to=4-4]
	\arrow["{<}"{marking, allow upside down}, Rightarrow, no head, from=4-4, to=4-5]
	\arrow[no head, from=4-5, to=4-6]
\end{tikzcd}
        \caption{The Distinguished weighted Dynkin diagrams of Levi subgroups of $F_4$}
        \label{fig:Levis of F_4}
    \end{figure}
    From this we directly see that all exposed roots have $\langle\alpha, \lambda \rangle =2$.

    In type $E_l$, there are three possibilities for the Levi subgroup types. Either there are only factors of type $A$ for which the result holds, or there is a unique factor of type $D_n$ and $n\leq 7$, or there is a factor of type $E_6$ or $E_7$. 
    In the case of a factor of type $E_6$ or $E_7$, the weighted Dynkin diagrams of distinguished parabolic subgroups (of which there are $3$ of type $E_6$ and $6$ of type $E_7$) are listed on page 176 of \cite{Carter85(93)}, from which it is clear that all exposed roots $\alpha$ have the desired property. See also Figures \ref{fig:E6} and \ref{fig:E7} in Section \ref{Section: distinguished dynkin diagrams}.
    
    This leaves only the case that $L$ has a factor of type $D_n$. All distinguished orbits of $D_n$ with $n\leq 7$ are listed in Figure \ref{fig:D_n} in Section \ref{Section: distinguished dynkin diagrams}. From this table we see that all exposed roots have the desired property.
\end{proof}

\begin{remark}
    Regarding the reference to the book \cite{Carter85(93)}, we briefly note that in the description of non-regular distinguished parabolic subgroups in type $B_l$ on page 175, one requires $k\geq 2$ (in the notation of the source) for the conditions to make sense, though this isn't explicit. In our application, the important fact is that for $B_l$ with $l=3$, the only distinguished orbit is the regular orbit. This can also be seen from the description of distinguished orbits in Theorem 8.2.14 of \cite{ColMcG93} via partitions of $2l+1=7$ into distinct odd parts.
\end{remark}

\begin{proof}[Proof of Theorem \ref{theorem -notsmooth}]
Consider a point $P=(\Phi_0,0)\in X_C$ with $\Phi_0\in T$. Define four subvarieties of $S_G$ that contain $P$ as follows:

\begin{enumerate}
    \item Let $\Oc=G.P$ be the $G$-orbit of $P$.
    \item Let $\tilde{T}$ be the maximal torus of $G$ seen as a closed subvariety of $S_G$ via the inclusion $\Phi \mapsto (\Phi,0)$. 
    \item Let $\Nc_0=\{N\in\mathfrak{g}: \Ad(\Phi_0)N=qN\}$ viewed as a closed subvariety of $S_G$ via the inclusion $N\mapsto (\Phi_0,N)$. 
    \item Let $U_0=U^-\cap\Stab_G(P)$ where $U^-$ is the opposite unipotent subgroup to $[B,B]$, viewed as a closed subvariety of $S_G$ via the inclusion $u\mapsto (\Phi_0u,0)$.
\end{enumerate}

\begin{claim}
    The Tangent spaces $T_P\Oc, T_P\tilde{T}, T_P\Nc_0$ and $T_PU_0$ form a direct sum inside $T_PS_G$. 
\end{claim}
Observe that that $T_P(S_G)\subseteq T_P(G\times \gf)= \gf\times \gf$. We briefly describe $T_P\Oc$ as a subspace of $\gf\times\{0\}$ (which we will conflate with $\gf$ as this shouldn't cause confusion). 

Consider the map $f:G\rightarrow G: g\mapsto g\Phi_0g^{-1}\Phi_0^{-1}$, comprised of the conjugation action of $g$ on $\Phi_0$ followed by right multiplication by $\Phi_0^{-1}$ (to ensure the identity is sent to the identity). Then $\Oc$ is isomorphic to the set theoretic image of $f$; a locally closed subscheme of $G$. 
The derivative of this map is $\id-\Ad(\Phi_0):\gf\rightarrow \gf$, and it factors through $\gf\twoheadrightarrow T_PC\hookrightarrow \gf$. We hence see a natural identification of $T_P\Oc$ with $\im(\id-\Ad(\Phi_0))$.

We proceed now to prove the claim.
Firstly, the intersection $T_P\tilde{T}\cap T_PU_0=\{0\}$ because $\Lie(\tilde{T})\cap \Lie(U^-)=\{0\}$.
Next, consider $T_P\Oc\cap (T_P\tilde{T}\oplus T_PU_0)$. 
Observe that $U_0, \tilde{T}\subseteq \Stab_G(\Phi)\subseteq G$ and that $T_P\Stab_G(\Phi_0)=\ker(\id-\Ad(\Phi_0))$. Then because $\Phi_0$ is semisimple, the intersection of $\im(\id-\Ad(\Phi_0))$ and $\ker(\id-\Ad(\Phi_0))$ is trivial (this is easily checked for $\GL_n$, and extended to all $G$ because $G$ can always be embedded into some $\GL_N$) and so we see $T_P(\Stab_G(\Phi_0))\cap T_P\Oc=0$, and thus
$T_P\Oc\cap (T_P\tilde{T}\oplus T_PU_0)=0$.

To show that $T_P\Nc_0$ intersects $T_P\Oc+T_P\tilde{T}+T_PU_0$ at the origin, it suffices to notice that an element of $T_PC+T_P\tilde{T}+T_PU_0$ takes the form $P'=(\Phi',0)$, while an element $P'\in T_P\Nc_0$ takes the form $P'=(\Phi_0,N)\in S_G(\F[\epsilon])$. 
For these to be equal, we must have $\Phi'=\Phi_0$ and $N=0$, so $P'=P$. This proves the claim.

Suppose that $C\subseteq \Nc_{G}$ is an adjoint orbit, neither zero nor a distinguished nilpotent orbit as in the hypothesis. Define $e\in C$, a choice of maximal torus and Borel $\tilde{T}\subset B$, the associated cocharacter $\lambda$, and minimal Levi $L\subset G$ all as in the general setup.
 Set $D=\tilde{T}\cap X_C$, and $\Nc_1=\Nc_0 \cap X_C$, and note that $U_0, \Oc\subseteq X_C$ already. Our aim is to show that there is a point $P\in X_C$ such that 
\[\dim_{\F}(T_P\Oc)+\dim_{\F}(T_PD)+\dim_{\F}(T_P\Nc_1)+\dim_{\F}(T_PU_0)>\dim(X_C)=\dim(G).\]

It is clear that whenever $z\in Z_L$ and $\mu\in \F$, the point $(z\lambda(\sqrt{q}),\mu e)\in X_C$. 
We choose the point $P=(\Phi_0,0)\in X_C$ with $\Phi_0=z\lambda(\sqrt{q})$ for some $z\in Z_L$, which we will determine momentarily. 

Regardless of the choice of $z$ for now, recall the decomposition 
\[\Lie(L)=\bigoplus_{i\in \Z}\gf_L(\lambda,i).\]
As $\gf_L(\lambda,0)$ is a Levi subalgebra of $\Lie(L)$, there is a Levi subgroup $M_0\subset L$ with $\Lie(M_0)=\gf_L(\lambda,0)$. It is clear that $M_0\subseteq \Stab_G(\Phi_0)$. 
Observe also that $\gf_L(\lambda,2)\subseteq \Nc_1$ and $Z_L \Phi_0\subseteq D$.
We define positive integers:
\begin{enumerate}
    \item $\epsilon_0:=\dim(\Stab_G(\Phi_0))-\dim_{\F}(\gf_L(\lambda,0))=\dim(G)-\dim(\Ocal)-\dim_{\F}(\gf_L(\lambda,0));$
    \item $\epsilon_1:=\dim_{\F}(T_PD)-\dim(Z_L)$;
    \item $\epsilon_2:=\dim_{\F}(T_P\Nc_1)-\dim_{\F}(\gf_L(\lambda,2))$;
    \item $\epsilon_3:=\dim_{\F}(T_PU_0)$.
\end{enumerate}
Putting this together, and using the fact that 
$$\dim(\gf_L(\lambda,0))=\dim(\gf_L(\lambda,2))+\dim(Z_L)$$ 
because $e$ is distinguished inside $L$ (this follows from a generalisation of Lemma 8.2.1 in \cite{MR1251060} to reductive groups in good characteristic), we see
\begin{align*}
    \dim(T_PX_C)
    &\geq\dim(T_P\Oc)+\dim(T_PD)+\dim(T_P\Nc_1)+\dim(T_PU_0) \\
    &\geq [\dim(G)-\dim(\gf_L(\lambda,0))-\epsilon_0]+[\dim(Z_L)+\epsilon_1]\\
    &+[\dim(\gf_L(\lambda,2))+\epsilon_2]+\epsilon_3 \\
    &=\dim(G)+\epsilon_1+\epsilon_2+\epsilon_3-\epsilon_0.
\end{align*}
Thus, it is enough to find some choice of $z\in Z_L$ such that $\epsilon_1+\epsilon_2+\epsilon_3-\epsilon_0>0$. 

Fix some root $\alpha\in\Delta\setminus \Delta_L$ adjacent to a root in $\Delta_L$ in the Dynkin diagram.
\begin{figure}[h!]
    \centering
    \includegraphics[width=0.45\textwidth]{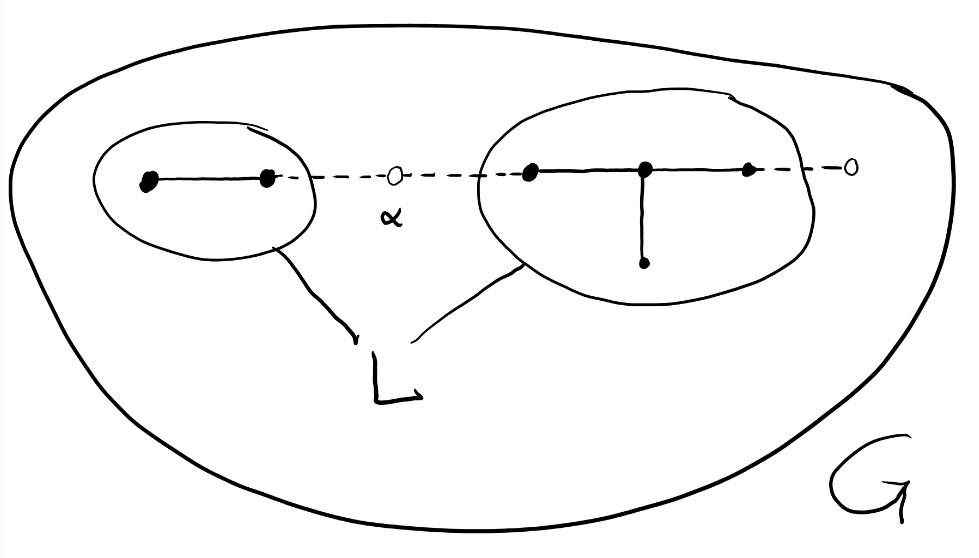}
    \caption{An example of $\alpha$ in the case where $G$ is of type $E_8$ and $L$ of type $D_4\times A_2$.}
    \label{fig:Algebraic group and Levi}
\end{figure}
There is a morphism of algebraic groups
\[B=\prod_{\beta\in \Delta}\beta:\tilde{T}\rightarrow \prod_{\beta\in\Delta}\G_m\]
which is surjective and has kernel $Z\subset \tilde{T}$. Because $B$ is surjective, we can choose some $z\in T$ such that $\alpha(z)=1$ and $\beta(z)=1$ whenever $\beta\in\Delta_L$ (so that $z\in Z_L$) and $\beta(z)=q\neq 1$ in all other cases. 
Consider $\Phi_0=z\lambda(\sqrt{q})$. Whenever $\gamma\in\Pi^+_G$ is a positive root of $G$, it decomposes as a product of simple roots $\gamma=\prod_{\beta\in\Delta}  \beta^{c_{\beta}}\in X(T)$ with $\sum_{\beta} c_{\beta}<h$. By design, each $\beta(\Phi_0)$ is either $1$ or $q$, so 
$\gamma(\Phi_0)\in\{1,q,q^2,...,q^{h-1}\}$, and $\gamma(\Phi_0)=1$ if and only if all the simple roots with $c_{\beta}\neq 0$ satisfy $\beta(\Phi_0)=1$.
\begin{claim}
    If $\gamma\in\Pi^+_G\setminus \Pi^+_L$ has $\gamma(\Phi_0)=1$, then $\gamma=\alpha$.
\end{claim}
 If $\gamma$ is simple, then $\gamma=\alpha$ by our choice of $\Phi_0$. Suppose for contradiction that $\gamma$ is not simple. Then it contains at least two simple roots in its decomposition, and one of these must be $\alpha$, as otherwise all simple roots are in $\Delta_L$ and $\gamma\in\Pi_L^+$. There must also be another root $\beta\in\Delta_L$ with $c_{\beta}\neq 0$, and for $\gamma$ to be a root, there must be a path from $\alpha$ to $\beta$ (in the Dynkin diagram) passing through vertexes $\beta'$ with $c_{\beta'}\neq 0$, and each of these as such (since $\gamma(\Phi_0)=1$) has $\beta'(\Phi_0)=1$.
But as $\alpha\in\Delta\setminus\Delta_L$ and $\beta\in\Delta_L$, at least one of the $\beta'$ is an exposed root, and thus $\beta'(\Phi_0)=q$ by Lemma \ref{Lem:exposed roots}. This is a contradiction. We conclude that $\gamma(\Phi_0)=1$ implies either $\gamma=\alpha$ or $\gamma\in\Pi^+_L$. This proves the claim. 

When $\beta\in\Pi$, denote the root subgroup of $\beta$ by $U_{\beta}\leq G$. 
By Theorem 3.5.3 of \cite{Carter85(93)}, the (connected) centraliser of $\Phi_0$ is
\[C_G(\Phi_0)^{\circ}=\langle \tilde{T},U_{\beta},U_{-\beta}: \beta(\Phi_0)=1 \rangle. \]
The subgroup generated by $\tilde{T}$ and all $U_{\beta}$ with $\beta(\Phi_0)=1$ and $\beta\in\Pi_L$ is simply $M_0^{\circ}$, so we see that 
\[C_G(\Phi_0)^{\circ}=\langle M_0^{\circ},U_{\alpha},U_{-\alpha}\rangle\]
and hence $\dim C_G(\Phi_0)=\dim(M_0)+2$ (or, in other words, $\epsilon_0=2$).

The reflection $s_{\alpha}\in N(\tilde{T})/\tilde{T}$ acts on $\tilde{T}$ and stabilises $\Phi_0$; thus it acts on $T_PX_C$. Further, this action preserves the subspaces $T_PD$ and $T_P \Nc_1$. However, since $\alpha$ is adjacent to a simple root of $L$, the reflection $s_{\alpha}$ does not preserve the Levi subgroup $L$, and hence preserves neither $Z_L$ nor $\gf_{L}(\lambda,i)$, and we see that $s_{\alpha}(Z_L)\neq Z_L$ and $s_{\alpha}(\gf_L(\lambda,2))\neq \gf_L(\lambda,2)$, so that $T_PD\supseteq s_{\alpha}(T_PZ_L)\cup T_PZ_L$ and
$T_P\Nc_1\supseteq s_{\alpha}(T_P\gf_L(\lambda,2))\cup T_P\gf_L(\lambda,2)$, forcing $\epsilon_1>0$ and $\epsilon_2>0$ respectively. 

For $\epsilon_3$, consider any choice of isomorphism $u_{-\alpha}:\G_a\xrightarrow{\sim}U_{-\alpha}$ and the adjoint action of $u_{-\alpha}(a)\in U_{-\alpha}$ on $e=\sum_{\beta} e_{\beta}\in \Lie(L)$. Because $\alpha$ is a simple root in $\Delta_G\setminus \Delta_L$, we see that 
\[[e_{-\alpha},\sum_{\beta}e_{\beta}]=\sum_{\beta}[e_{-\alpha},e_{\beta}]=0\]
and thus, that $\Ad(u_{-\alpha}(a))e=e$. Hence,
\[\big\{(\Phi_0u_{-\alpha}(a),\mu e):a\in \G_a, \mu\in \G_m\big\}\]
 is a locally open subscheme of $p^{-1}(C)$, from which we see that $(\Phi_0u_{-\alpha}(\epsilon),0)$ is a deformation in $T_PU_0$, forcing $\epsilon_3>0$.

 We then obtain the inequality $\epsilon_1+\epsilon_2+\epsilon_3-\epsilon_0\geq 3-2=1$, proving that $(\Phi_0,0)$ is a singular point of $X_C$.
\end{proof}

\begin{example}
    Consider the group
    \[G=\GSP_4(R)=\{M\in \GL_4(R): M\Omega M^{-1}=\lambda\Omega \textrm{ for some } \lambda\in \G_m(R)\}\]
    with the convention where $\Omega=\left(\begin{smallmatrix}
        & & & 1 \\
        & & 1 & \\
        & -1 & & \\
        -1 & & & 
    \end{smallmatrix}\right)$ so that a Borel subgroup can be given by the intersection of $\GSP_4$ with the upper triangular matrices in $\GL_4$. We let $L=GL_2\subseteq G$ be the Levi subgroup corresponding to the short root. Then $e=\left(\begin{smallmatrix}
       0 & 1 & &  \\
        & 0 & 0 & \\
        &  & 0 & -1 \\
         & & &  0
    \end{smallmatrix}\right)$ is distinguished in $L$ and the associated cocharacter is $\lambda(t)=\diag(t,t^{-1},t,t^{-1})$. We choose $\Phi_0=\diag(q,1,1,q^{-1})$ and $\alpha$ to be the root corresponding to the one-parameter subgroup defined as $U_{\alpha}=\left\{\left(\begin{smallmatrix}
       1 &  & &  \\
        & 1 & * & \\
        &  & 1 &  \\
         & & &  1
    \end{smallmatrix}\right)\right\}$. 
    Explicitly, we now see that $\Stab_G(P)=\G_m\times \GL_2$. We can also describe the subvarieties:
    \begin{align*}
        Z_L&=\Bigg\{\bigg(\left(\begin{smallmatrix}
       ab &  & &  \\
        & ab &  & \\
        &  & ab^{-1} &  \\
         & & &  ab^{-1}
    \end{smallmatrix}\right),0\bigg):a,b\in\G_m\Bigg\}, \\ 
        s_{\alpha}(Z_L)&=\Bigg\{\bigg(\left(\begin{smallmatrix}
       ab &  & &  \\
        & ab^{-1} &  & \\
        &  & ab &  \\
         & & &  ab^{-1}        
    \end{smallmatrix}\right),0\bigg):a,b\in\G_m\Bigg\} \\
        \gf_L(\lambda,2) &= \bigg\{\Big(\Phi_0,\left(\begin{smallmatrix}
       0 & a & &  \\
        & 0 &  & \\
        &  & 0 & -a \\
         & & &  0
    \end{smallmatrix}\right)\Big):a\in\G_a\bigg\}\\
        s_{\alpha}(\gf_L(\lambda,2)) &=\bigg\{\Big(\Phi_0,\left(\begin{smallmatrix}
       0 &  & a &  \\
        & 0 &  & a \\
        &  & 0 &  \\
         & & &  0
    \end{smallmatrix}\right)\Big):a\in\G_a\bigg\} \\
        U_0 &= \bigg\{\Big(\left(\begin{smallmatrix}
       q &  & &  \\
        & 1 &  & \\
        & a & 1 &  \\
         & & &  q^{-1}
    \end{smallmatrix}\right),0\Big):a\in\G_a\bigg\}
    \end{align*}

    When we put all this together, we see the contribution from $D,\Nc_1$ and $U_0$ is $6$-dimensional, and thus 
    \[\dim T_PX_C\geq\dim(\GSP_4)-\dim(\Stab_G(P))+6=\dim(\GSP_4)+1.\] 
\end{example}
We can piece Theorems \ref{theorem-smooth} and \ref{theorem -notsmooth} together to make the statement:

\begin{corollary}
    Let $G$ be a connected reductive group over a field $\F$ of characteristic $0$ or $l$, and suppose $G^{\ad}=\prod_i G_i$ where each $G_i$ has a (almost) simple derived subgroup. Suppose $q$ is considerate towards each $G_{i,\F}$. 
    Then the smooth irreducible components of $S_{G^{\ad}}$ are precisely those of the form $\prod_i X_i$ where each $X_i\subseteq S_{G_i}$ is a smooth irreducible component of $S_{G_i}$. That is, each $X_i$ corresponds to a distinguished nilpotent orbit of $G_i$ or the zero orbit. The smooth components of $S_G$ are precisely preimages the smooth components of $S_{G^{\ad}}$ under the obvious map $S_G\rightarrow S_{G^{\ad}}$.
\end{corollary}

\subsection{Distinguished orbits in type $D$ and $E$}{\label{Section: distinguished dynkin diagrams}}

For the convenience of the reader in understanding Lemma \ref{Lem:exposed roots}, we have included a list of weighted Dynkin diagrams for all distinguished orbits in types $D_n$ and $E_n$ with $n\leq 7$.
    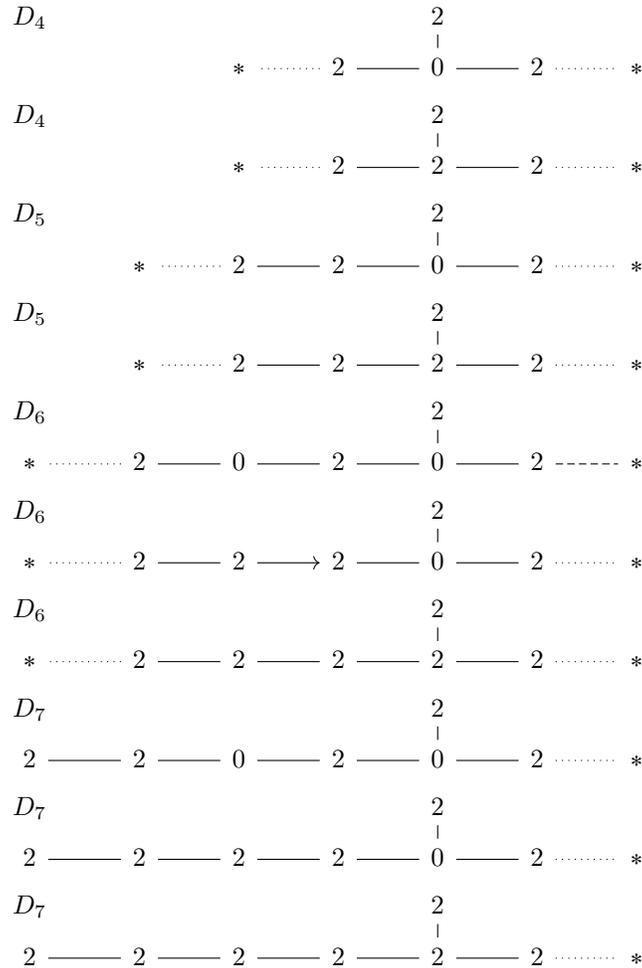
\begin{figure}[h!]
        \centering
        \begin{tikzcd}[row sep=0.3em]
	{D_4} &&&& 2 \\
	&& {*} & 2 & 0 & 2 & {*} \\
	{D_4} &&&& 2 \\
	&& {*} & 2 & 2 & 2 & {*} \\
	{D_5} &&&& 2 \\
	& {*} & 2 & 2 & 0 & 2 & {*} \\
	{D_5} &&&& 2 \\
	& {*} & 2 & 2 & 2 & 2 & {*} \\
	{D_6} &&&& 2 \\
	{*} & 2 & 0 & 2 & 0 & 2 & {*} \\
	{D_6} &&&& 2 \\
	{*} & 2 & 2 & 2 & 0 & 2 & {*} \\
	{D_6} &&&& 2 \\
	{*} & 2 & 2 & 2 & 2 & 2 & {*} \\
	{D_7} &&&& 2 \\
	2 & 2 & 0 & 2 & 0 & 2 & {*} \\
	{D_7} &&&& 2 \\
	2 & 2 & 2 & 2 & 0 & 2 & {*} \\
	{D_7} &&&& 2 \\
	2 & 2 & 2 & 2 & 2 & 2 & {*}
	\arrow[dotted, no head, from=2-4, to=2-3]
	\arrow[no head, from=2-4, to=2-5]
	\arrow[no head, from=2-5, to=1-5]
	\arrow[no head, from=2-5, to=2-6]
	\arrow[dotted, no head, from=2-6, to=2-7]
	\arrow[dotted, no head, from=4-4, to=4-3]
	\arrow[no head, from=4-4, to=4-5]
	\arrow[no head, from=4-5, to=3-5]
	\arrow[no head, from=4-5, to=4-6]
	\arrow[dotted, no head, from=4-6, to=4-7]
	\arrow[dotted, no head, from=6-3, to=6-2]
	\arrow[no head, from=6-4, to=6-3]
	\arrow[no head, from=6-4, to=6-5]
	\arrow[no head, from=6-5, to=5-5]
	\arrow[no head, from=6-5, to=6-6]
	\arrow[dotted, no head, from=6-6, to=6-7]
	\arrow[dotted, no head, from=8-3, to=8-2]
	\arrow[no head, from=8-3, to=8-4]
	\arrow[no head, from=8-4, to=8-5]
	\arrow[no head, from=8-5, to=7-5]
	\arrow[no head, from=8-5, to=8-6]
	\arrow[dotted, no head, from=8-6, to=8-7]
	\arrow[no head, from=9-5, to=10-5]
	\arrow[dotted, no head, from=10-2, to=10-1]
	\arrow[no head, from=10-3, to=10-2]
	\arrow[no head, from=10-4, to=10-3]
	\arrow[no head, from=10-5, to=10-4]
	\arrow[no head, from=10-5, to=10-6]
	\arrow[dashed, no head, from=10-6, to=10-7]
	\arrow[no head, from=11-5, to=12-5]
	\arrow[dotted, no head, from=12-2, to=12-1]
	\arrow[no head, from=12-2, to=12-3]
	\arrow[from=12-3, to=12-4]
	\arrow[no head, from=12-5, to=12-4]
	\arrow[no head, from=12-5, to=12-6]
	\arrow[dotted, no head, from=12-6, to=12-7]
	\arrow[no head, from=13-5, to=14-5]
	\arrow[dotted, no head, from=14-2, to=14-1]
	\arrow[no head, from=14-3, to=14-2]
	\arrow[no head, from=14-4, to=14-3]
	\arrow[no head, from=14-5, to=14-4]
	\arrow[no head, from=14-5, to=14-6]
	\arrow[dotted, no head, from=14-6, to=14-7]
	\arrow[no head, from=15-5, to=16-5]
	\arrow[no head, from=16-2, to=16-1]
	\arrow[no head, from=16-3, to=16-2]
	\arrow[no head, from=16-4, to=16-3]
	\arrow[no head, from=16-5, to=16-4]
	\arrow[no head, from=16-5, to=16-6]
	\arrow[dotted, no head, from=16-6, to=16-7]
	\arrow[no head, from=17-5, to=18-5]
	\arrow[no head, from=18-1, to=18-2]
	\arrow[no head, from=18-2, to=18-3]
	\arrow[no head, from=18-3, to=18-4]
	\arrow[no head, from=18-4, to=18-5]
	\arrow[no head, from=18-5, to=18-6]
	\arrow[dotted, no head, from=18-6, to=18-7]
	\arrow[no head, from=20-1, to=20-2]
	\arrow[no head, from=20-2, to=20-3]
	\arrow[no head, from=20-3, to=20-4]
	\arrow[no head, from=20-4, to=20-5]
	\arrow[no head, from=20-5, to=19-5]
	\arrow[no head, from=20-5, to=20-6]
	\arrow[dotted, no head, from=20-6, to=20-7]
\end{tikzcd}
        \caption{The weighted Dynkin diagrams of distinguished orbits of type $D_n$}
        \label{fig:D_n}
    \end{figure}
   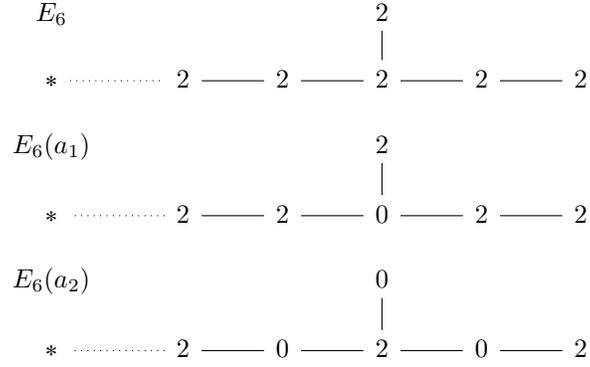
\begin{figure}[h!]
        \centering
\begin{tikzcd}[row sep=1.0em]
	{E_6} &&& 2 \\
	{*} & 2 & 2 & 2 & 2 & 2 \\
	{E_6(a_1)} &&& 2 \\
	{*} & 2 & 2 & 0 & 2 & 2 \\
	{E_6(a_2)} &&& 0 \\
	{*} & 2 & 0 & 2 & 0 & 2
	\arrow[dotted, no head, from=2-1, to=2-2]
	\arrow[no head, from=2-2, to=2-3]
	\arrow[no head, from=2-3, to=2-4]
	\arrow[no head, from=2-4, to=1-4]
	\arrow[no head, from=2-4, to=2-5]
	\arrow[no head, from=2-5, to=2-6]
	\arrow[no head, from=3-4, to=4-4]
	\arrow[dotted, no head, from=4-2, to=4-1]
	\arrow[no head, from=4-3, to=4-2]
	\arrow[no head, from=4-4, to=4-3]
	\arrow[no head, from=4-4, to=4-5]
	\arrow[no head, from=4-5, to=4-6]
	\arrow[no head, from=5-4, to=6-4]
	\arrow[dotted, no head, from=6-2, to=6-1]
	\arrow[no head, from=6-3, to=6-2]
	\arrow[no head, from=6-4, to=6-3]
	\arrow[no head, from=6-4, to=6-5]
	\arrow[no head, from=6-5, to=6-6]
\end{tikzcd}
        \caption{The weighted Dynkin diagrams of distinguished orbits of type $E_6$}
        \label{fig:E6}
    \end{figure}
   \begin{figure}[h!]
        \centering
\begin{tikzcd}[row sep=1.0em]
	{E_7} &&&& 2 \\
	{*} & 2 & 2 & 2 & 2 & 2 & 2 \\
	{E_7(a_1)} &&&& 2 \\
	{*} & 2 & 2 & 2 & 0 & 2 & 2 \\
	{E_7(a_2)} &&&& 2 \\
	{*} & 2 & 0 & 2 & 0 & 2 & 2 \\
	{E_7(a_3)} &&&& 0 \\
	{*} & 2 & 2 & 0 & 2 & 0 & 2 \\
	{E_7(a_4)} &&&& 0 \\
	{*} & 2 & 0 & 0 & 2 & 0 & 2 \\
	{E_7(a_5)} &&&& 0 \\
	{*} & 2 & 0 & 0 & 2 & 0 & 0
	\arrow[no head, from=1-5, to=2-5]
	\arrow[dotted, no head, from=2-1, to=2-2]
	\arrow[no head, from=2-2, to=2-3]
	\arrow[no head, from=2-3, to=2-4]
	\arrow[no head, from=2-4, to=2-5]
	\arrow[no head, from=2-5, to=2-6]
	\arrow[no head, from=2-6, to=2-7]
	\arrow[no head, from=3-5, to=4-5]
	\arrow[dotted, no head, from=4-2, to=4-1]
	\arrow[no head, from=4-3, to=4-2]
	\arrow[no head, from=4-4, to=4-3]
	\arrow[no head, from=4-4, to=4-5]
	\arrow[no head, from=4-5, to=4-6]
	\arrow[no head, from=4-6, to=4-7]
	\arrow[no head, from=5-5, to=6-5]
	\arrow[dotted, no head, from=6-2, to=6-1]
	\arrow[no head, from=6-3, to=6-2]
	\arrow[no head, from=6-4, to=6-3]
	\arrow[no head, from=6-4, to=6-5]
	\arrow[no head, from=6-5, to=6-6]
	\arrow[no head, from=6-6, to=6-7]
	\arrow[no head, from=7-5, to=8-5]
	\arrow[dotted, no head, from=8-2, to=8-1]
	\arrow[no head, from=8-3, to=8-2]
	\arrow[no head, from=8-4, to=8-3]
	\arrow[no head, from=8-5, to=8-4]
	\arrow[no head, from=8-5, to=8-6]
	\arrow[no head, from=8-6, to=8-7]
	\arrow[no head, from=9-5, to=10-5]
	\arrow[dotted, no head, from=10-2, to=10-1]
	\arrow[no head, from=10-3, to=10-2]
	\arrow[no head, from=10-4, to=10-3]
	\arrow[no head, from=10-5, to=10-4]
	\arrow[no head, from=10-5, to=10-6]
	\arrow[no head, from=10-6, to=10-7]
	\arrow[no head, from=11-5, to=12-5]
	\arrow[dotted, no head, from=12-2, to=12-1]
	\arrow[no head, from=12-3, to=12-2]
	\arrow[no head, from=12-4, to=12-3]
	\arrow[no head, from=12-5, to=12-4]
	\arrow[no head, from=12-5, to=12-6]
	\arrow[no head, from=12-6, to=12-7]
\end{tikzcd}
        \caption{The weighted Dynkin diagrams of distinguished orbits of type $E_7$}
        \label{fig:E7}
    \end{figure}
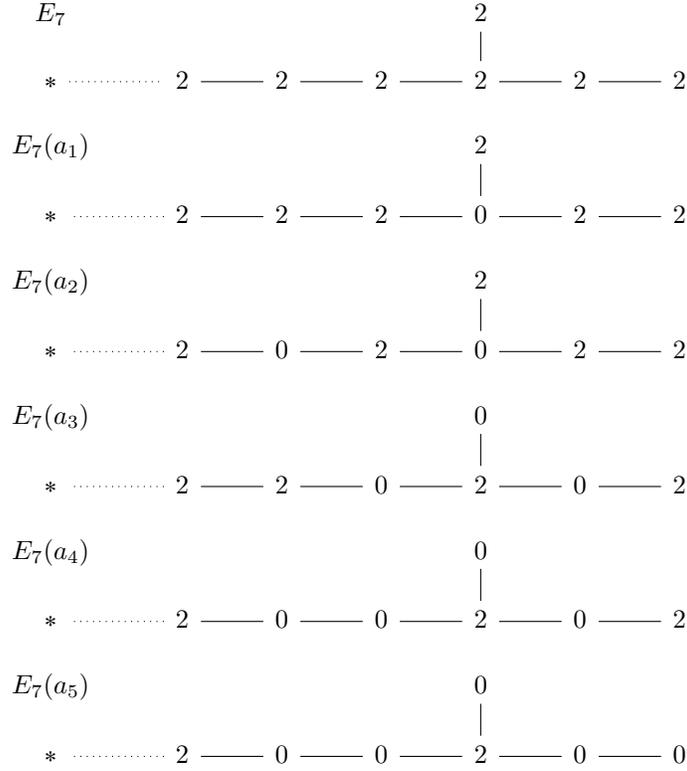

\newpage \section{Automorphic forms for unitary groups}

We now turn to an application of the smoothness result found in Section 3. In this section, we define the space of ordinary automorphic forms and the Hecke algebra attached to it. We then state a freeness result, which we will prove in the final section of this paper.

Let $l$ be a prime. Suppose $F^+$ is a totally real number field with an imaginary quadratic extension $F$ such that all primes $v$ of $F^+$ above $l$ split in $F$. 
Let $S_l$ be the set of all primes of $F^+$ that lie above $l$. 
Let $L$ be a finite extension of $\Q_l$ with ring of integers $\Ocal$ and residue field $\F$. Let $\bar{L}$ be a choice of algebraic closure. We will assume that $L$ is large enough that all embeddings $F\hookrightarrow \bar{L}$ lie inside $L$.
Let $c\in \Gal(F/F^+)$ be the unique non-trivial element, given by complex conjugation. For $a\in F$, we will denote $c(a)$ by $\bar{a}$ when convenient. 

\subsection{Unitary groups}

Consider $D/F$ a central simple algebra of $F$-dimension $n^2$, and let $S_D$ be a finite set of primes of $F^+$ that split in $F$.
Suppose that 

\begin{itemize}
    \item $D$ splits at all places $w$ of $F$ that do not lie above any place in $S_D$;
    \item There is an isomorphism $D^{op}\cong D\otimes_{F,c}F$ of $F$-algebras;
    \item The intersection $S_D\cap S_l$ is empty; 
    \item $D_w$ is a division algebra at all places $w$ of $F$ above a place in $S_D$;
    \item Either $n$ is odd, or $n$ is even and $\frac{n}{2}[F^+:\Q]+\#S_D\equiv 0 \textrm{ } (\textrm{mod }2)$.
\end{itemize}
Because of conditions 1 (note that all places in $S_D$ split), 2 and 5, 
we can find an involution of the second kind on $D$ by \cite{CHT08} Section 3.3 (p.95). That is, we may construct a map

$$^*:D\rightarrow D$$
such that:
\begin{itemize}
    \item $^*$ is an $F^+$ linear anti-automorphism of $D$;
    \item $(a^*)^*=a$ for all $a\in D$;
    \item The involution $^*$ coincides with complex conjugation when restricted to $F$. 
\end{itemize}

In addition, we assume that this involution of the second kind is positive. That is, for any $\gamma\in D\setminus \{0\}$,
$$\tr_{F:\Q}[\tr_{D/F}(\gamma\gamma^*)]>0.$$
Such an involution gives rise to a positive Hermitian form $\langle,\rangle:D\times D \rightarrow D$ given by $\langle x,y\rangle =x^*y$.

Let $\Ocal_D$ be an order in $D$ such that $\Ocal_D^*=\Ocal_D$ and such that $\Ocal_{D,v}$ is a maximal order of $D_v$ for any split prime $v$ of $F^+$, as in Section 3.3 of \cite{CHT08}.
Define the unitary group over $\Ocal_{F^+}$ whose $R$-points (where $R$ is an $\Ocal_{F^+}$-algebra) are given by $G_D(R)=\{g\in (\Ocal_D\otimes_{\Ocal_{F^+}} R)^{\times}:g^*g=1\}$. 
Then $G_D$ is an algebraic group over $\Ocal_{F^+}$.
By the positivity condition, we have $G_{D,v}\cong U(n)$ at each infinite place $v$ of $F^+$.

For each prime $v$ of $F^+$ that splits in $F$, choose a prime $\tilde{v}$ of $F$ lying above $v$. This choice allows us to give an isomorphism $i_{\tilde{v}}:G_D(F^+_v)\xrightarrow{~}(D\otimes_F F_{\tilde{v}})^\times$ which restricts to an isomorphism $G_D(\Ocal_{F^+,v})\cong (\Ocal_{D,\tilde{v}})^{\times}$ as in Section 3.3 of \cite{CHT08}. 
Note that when $v\notin S_D$ is split in $F$, then $G_D$ is split, so that $G_D(F^+_v)\cong (D\otimes_{F^+}F^+_v)^{\times} \cong (D\otimes_F F_{\tilde{v}})^{\times}\cong \GL_n(F_{\tilde{v}})$.
If $T$ is a set of primes of $F^+$ that splits in $F$, set $\tilde{T}=\{\tilde{v}|v\in T\}$.

\subsection{Automorphic forms of $G_D$}
We define the automorphic forms for $G_D$ as in \cite{gross_1999} and \cite{CHT08}. To do this, we recall some facts from the representation theory of reductive groups.

Let $G$ be a split reductive group defined over $L$, and let $T\subseteq B \subseteq G $ be a choice of split maximal torus and Borel subgroup of $G$.
Recall that finite-dimensional simple modules of $G$ are uniquely determined by their highest weight in the character group of the torus $X^{\bullet}(T):=\Hom(T,\G_m)$, and that such a representation exists if and only if this highest weight $\nu$ lies in a dominant Weyl chamber. 

In the case of $\GL_n$ and the standard upper Borel subgroup and maximal torus (defined over $L$), the set of weights naturally corresponds to $\Z^n$, and the set of dominant weights is $\Z^n_+:=\{\nu=(\nu_1,...,\nu_n)\in \Z^n: \nu_i\geq \nu_{i+1} \forall i\}$.
We set the $L$-vector space $W_{\nu}$ to be the irreducible representation of highest weight $\nu$. We will need to choose a $\Ocal$ lattice of $W_{\nu}$, which we do in the style of Section 1.1 of \cite{geraghty_2018} as follows. 
Note that $\GL_n$, $B$, and $T$ are defined over $\Ocal$. For a dominant weight $\nu$, set $\xi_{\nu}$ to be the induced representation $\textrm{Ind}_{B}^{\GL_n}(w_0(\nu))_{/\Ocal}$ of the algebraic group $\GL_{n/\Ocal}$ defined as the functor whose $R$ points 
\[\Ind_B^G(w_0(\nu)):=\big\{f\in R[\GL_n]: f(bg)=w_0(\nu)(b).f(g) \forall g\in \GL(R), b\in B(R) \big\},\] 
where $w_0$ the longest element of the Weyl group. By Proposition II.2.2 and Corollary II.5.6 of \cite{Jantzen}, the representation $\xi_{\nu}$ is irreducible of highest weight $\nu$. We denote by $M_{\nu}$ the representation given by the $\Ocal$-points of $\xi_{\nu}$, so that $ M_{\nu}\otimes_{\Ocal}L\cong \xi_{\nu}(L)\cong W_{\nu}$.
\begin{remark}
    We remark that $w_0$ arises here because our convention chooses $B$ as the Borel of \emph{upper} triangular matrices, whereas Jantzen induces from the Borel of \emph{lower} triangular matrices. These two choices of Borel subgroup are related by $w_0$. 
\end{remark}

The finite-dimensional algebraic representations in $L$ vector spaces of $G_{D,F^+_l}\cong \prod_{w\in\tilde{S_l}}\GL_{n,F_w}$ are characterised by the sequence of dominant weights, one for each embedding corresponding to $w\in \tilde{S_l}$.
We define the set as $W=(\Z^n_+)^{\Hom(F^+,L)}$. 
For each $\mu\in W$, we can now define the algebraic representation of $G_{D/\Ocal_{F^+}}$ with highest weight $\mu$ by $M_{\mu}=\bigotimes_{\tau\in\Hom(F^+,L),\Ocal}M_{\nu_{\tau}}$, and $W_{\mu}=M_{\mu}\otimes_{\Ocal} L$. 

For each $v\in S_D$, choose a finite-free $\Ocal$-module representation $\rho_v:G_D(\Ocal_{F^+,v})\rightarrow \GL(M_v)$ with open kernel such that $M_v\otimes \bar{L}$ is irreducible. Set $M_{\{\rho_v\}}=\bigotimes_{v\in S_D}M_v$. We set $M_{\mu,\{\rho_v\}}=M_{\mu}\otimes M_{\{\rho_v\}}$. 

\begin{definition}
Let $\lambda=(\mu,\{\rho_v\})$ be as above. We define the space of automorphic forms for $G_D$ of weight $\lambda$ with $A$-coefficients $S_{\lambda}(A)$, where $A$ is an $\Ocal$-module, as the space of functions 

\[f:G_D(F^+)\backslash G_D(\A)\rightarrow M_{\lambda}\otimes_{\Ocal} A\]
such that there is an open compact subgroup 
$$U\subset G_D(\mathbb{A}^{\infty,S_l}_{F^+})\times G_D(\Ocal_{F^+,l})$$
with 
$$u|_{S_l\cup S_D} f(gu)=f(g)$$
for all $g\in G_D(\A)$ and $u\in U$ where $u|_{S_l\cup S_D}$ denotes the action of $u$ on $M_{\lambda}$ factoring through $\prod_{v\in S_D\cup S_l}G_D(F^+_v)$.
\end{definition}

Notice that $S_{\lambda}(A)$ is a smooth representation of $G_D(\A)$ under the action 
$$(h.f)(g)=h|_{S_l\cup S_D} f(gh).$$
(Again, $h|_{S_l\cup S_D}$ acts through the representation of $G_D(F^+_l)\times \prod_{v\in S_D}G_D(F^+_v)$ on $M_{\lambda}$). We denote by $S_{\lambda}(U,A)=S_{\lambda}(A)^U$ the invariants under this action.

\subsection{Hecke Operators}

For much of what remains, the argument will be a slight adaptation on that in \cite{geraghty_2018}, the important details of which can be found in Sections 2 and 4.
Let $T$ be a finite set of places of $F^+$ containing $S_D\cup S_l$ such that all places in $T$ split in $F$, and let $\tilde{T}$ be a set of primes of $F$ above those in $T$ as defined before.
Fix an open compact subgroup $U=\prod_{v} U_{v}$ of $G_D(\A)$ such that $U_v$ is hyperspecial at all places $v$ outside $T$.
Suppose further that $U$ is sufficiently small; that is, there is a place $v$ such that $U_v$ contains no non-identity finite order elements. 
We define the Hecke operators on the subspace $S_{\lambda}(U,A)$.

\paragraph{Hecke operators at unramified places}
Let $v$ be a place of $F^+$ split in $F$ and $\tilde{v}$ be a place in $F$ over $v$. Let $\varpi_{\tilde{v}}$ be a uniformiser.
We can define the Hecke operators as the double coset operators:

$$T_v^{(i)}=\left[i_{v}^{-1}\left(\GL_n(\mathcal{O}_{F,\tilde{v}})\begin{pmatrix} \varpi_{\tilde{v}} I_i & 0 \\ 0 & I_{n-i} \end{pmatrix} \GL_n(\mathcal{O}_{F,\tilde{v}})\right)\times U^{v}\right]$$

\paragraph{Hecke operators at places dividing $l$.}

At places dividing the residual characteristic of $\mathcal{O}$, we set $\alpha_{\Tilde{v}}^{(i)}=\begin{pmatrix} \varpi_{\Tilde{v}} I_i & 0 \\ 0 & I_{n-i} \end{pmatrix}$, and define
$$U_{\mu,{\Tilde{v}}}^{(i)}=(w_0\mu_v)(\alpha_{\Tilde{v}}^{(i)})^{-1}[U\alpha_{\Tilde{v}}^{(i)} U ]$$
where $w_0$ is the longest element of the Weyl group of $\GL_n$ and $\mu=(\mu_v)\in W$ with $\mu_v$ the dominant weight for the corresponding embedding $F^+\hookrightarrow L$.

We make the following adjustment to the group $U$. 

\begin{definition}
For $v$ a place of $F^+$ above $l$ and $b$ a positive integer, let $I^{b}(\Tilde{v})$ be the set of matrices in $\GL_n(\Ocal_{F_{\Tilde{v}}})$ which are upper-triangular and unipotent \textrm{mod} $\Tilde{v}^b$. 
Define $U(l^{b})=\prod_{v \in S_l}i_v^{-1}(I^{b}(\Tilde{v}))\times U^{l}\subseteq G_D(\A)$ where $U^l$ denotes the product $\prod_{v\nmid l}U_v$. 
\end{definition}
In the case with the group $U(l^{b})$, we further define the following diamond operators:
\begin{definition} Let $T_n$ be the maximal torus inside $\GL_n$ as before.
    For $v\in S_l$, and $u\in T_n(\mathcal{O}_{F_{\Tilde{v}}})$, define $\langle u\rangle$ as the operator 

$$[U(l^{b})uU(l^{b})]$$
on $S_{\lambda}(U(l^{b}),A)$. For $u\in T_n(\mathcal{O}_{F^+,l})=\prod_{v\in S_l}T_n(\mathcal{O}_{F_{v}})\cong\prod_{v\in S_l}T_n(\mathcal{O}_{F_{\Tilde{v}}})$, define $\langle u\rangle=\prod_{v \in S_l}\langle u_{\tilde{v}}\rangle$.

\end{definition}

Let $A$ be an $\Ocal$-algebra and $M$ an $A$-module.
Define the Hecke algebra $\mathbb{T}^T=\mathbb{T}^T(U(l^{b}),M)$ as the $A$-subalgebra of 
$\End(S_{\lambda}(U(l^{b}),M))$ generated by all the operators 
\[\{T_{\tilde{v}}^{(i)},(T_{\tilde{v}}^{(n)})^{-1})| v\notin T, \textrm{  split in } F\}\cup\{U_{\mu,{\Tilde{v}}}^{(i)}| v\in S_l\}\cup\{\langle u\rangle|u\in T_n(\Ocal_{F^+,l})\}.\]

Notice that the map $u\mapsto \langle u \rangle$ defines a group homomorphism 
$$T_n(\mathcal{O}_{F^+,l})\rightarrow \mathbb{T}^T(U(l^{b}),M)^{\times}$$
which factors through $T_n(\mathcal{O}_{F^+,l}/l^b)=\prod_{v\in S_l}T_n(\mathcal{O}_{F^+,v}/v^b)$.

\subsection{Big ordinary Hecke algebras and the action of $\Lambda$}

From this point on, we wish to focus on the cases where $A\in\cat{Mod}_{\Ocal}$ is one of $\Ocal, L/\Ocal$ or a finite module $\Ocal/\pi^n\Ocal$. 

Recall from Hida theory, as fully explained in Section 2.4 of \cite{geraghty_2018}, that for any place $v\in S_l$ and any $i$, the operator $e^{(i)}_{v}:=\lim_{n\rightarrow \infty} (U_{\mu,{\Tilde{v}}}^{(i)})^{n!}$ is a projection on $S_{\lambda}(U,A)$. 
We can further define the projection $e=\prod_{v,i} e^{(i)}_{v}$. 
We define the ordinary submodule $S^{\text{ord}}_{\lambda}(U,A):=e.S_{\lambda}(U,A)$ as the image of this projection. 
Notice, since all the Hecke operators commute, that this is a Hecke invariant submodule. We also define $\mathbb{T}^{T,\text{ord}}(U(l^{b}),A)=e\mathbb{T}^T(U(l^{b}),A)$.

\begin{definition}
Let $T_n$ be the maximal torus of $\GL_n$ as before. For $b\geq 1$, let $T_n(l^b)$ be the kernel of $T_n(\mathcal{O}_{F^+,l})\rightarrow T_n(\mathcal{O}/l^b)$. 
We define $\Lambda$ as the algebra 
$$\Lambda=\mathcal{O}[[T_n(l)]]=\varprojlim_{b'\geq 1} \mathcal{O} [T_n (l) / T_n (l^{b'})].$$
\end{definition}

We denote by $a_N$ the kernel of the map $\Lambda\rightarrow \mathcal{O}[T_n(l)/T_n(l^N)]$. 
Since $U$ is sufficiently small, we see $S^{\text{ord}}_{\lambda}(U(l^{b,c}),A)$ is a free $\Lambda/a_b$-module, through the action of $T_n(\mathcal{O}_{F^+,l})$, and hence we have an inclusion $\Lambda/a_b \hookrightarrow \mathbb{T}(U(l^{b}),L/\Ocal)$ by Proposition 2.20 of \cite{geraghty_2018}.

\subsubsection{Infinite level}

We need to consider the big ordinary Hecke algebra. Set 

$$\mathbb{T}^{T,\ord}(U(l^{\infty}),A)=\varprojlim_{b>0} \mathbb{T}^{T,\ord}(U(l^{b}),A)$$
and 

$$S^{\ord}(U(l^{\infty}),A)=\varinjlim_{b>0} S^{\ord}(U(l^{b}),A).$$
Notice that because of the inclusions $\Lambda/a_b\hookrightarrow \T^{T,\ord}(U(l^{b,c}),L/\Ocal)$, we get an inclusion $\Lambda \hookrightarrow \T^{T,\ord}(U(l^{\infty}),L/\Ocal)$, and we see that $S^{\ord}(U(l^\infty),L/\Ocal)$ is a discrete $\Lambda$-module, so its Pontryagin dual is a compact $\Lambda$-module. (and in fact is finite free, by Proposition 2.20 of \cite{geraghty_2018} since we assume $U(l)$ is sufficiently small.)

We can now give a statement of a theorem that we prove by the application of Theorem \ref{theorem-smooth}. Under certain hypotheses (to be determined in Section 5), we have Theorem \ref{theorem locfree}, which states:

\begin{namedtheorem}{Theorem \ref{theorem locfree}}
    The $\mathbb{T}^{T,\ord}(U(l^{\infty}),L/\Ocal)$-module $S^{\ord}(U(l^{\infty}),L/\Ocal)^{\vee}$ is locally free over the generic fibre $\mathbb{T}^{T,\ord}(U(l^{\infty}),L/\Ocal)[1/l]$.
\end{namedtheorem}
    
As a consequence, the multiplicity of $S^{\ord}(U(l^{\infty}),L/\Ocal)^{\vee}$ is the same at every characteristic zero point of $\mathbb{T}^{T,\ord}(U(l^{\infty}),L/\Ocal)$, and thus we expect the multiplicity of non-classical points (those corresponding to Hida families of ordinary automorphic forms) is the same as at classical automorphic forms in $S_{\lambda}(U,A)$. 

\section{Galois representations and deformation rings}

\subsection{Local deformation rings}

In this section, we let $G_{F^+}$ and $G_F$ be the absolute Galois groups of $F^+$ and $F$ respectively, and $G_{F^+,v}$, $G_{F,w}$ be the decomposition groups at the places $v,w$ of $F^+$ and $F$, respectively.

We now define a deformation problem. Let $v\in S_D$ with residue field of size $q_v$, and let $X_{\textrm{St}}\subseteq S_{\GL_n}$ be the irreducible component corresponding to the regular nilpotent orbit. We say that an $n$-dimensional representation $\rho:G_{F^+,v}\rightarrow \GL_n(A)$ is Steinberg if the representation $\rho$ lies in the $A$-points of this irreducible component $X_{\textrm{St}}$.
When $A$ is a characteristic 0 field and $WD(\rho)=(r,N)$ is the Weil-Deligne representation obtained from $\rho$, then this condition is equivalent to the condition $r$ being unramified and the eigenvalues of $r(\Frob_{q_v})$ are in the ratio $q_v^{n-1}:q_v^{n-2}:...:q_v:1$.

Let $C_{\Ocal}$ be the category of local Artinian $\Ocal$-algebras with residue field $\F$, (that is, the category of coefficient rings as defined in Mazur's article in \cite{CSS97}). 
For each $v\in S_D$ and Steinberg representation $\bar{\rho}_v:G_{F,\tilde{v}}\rightarrow \GL_n(\F)$ define a functor 
\begin{align*}
    D^{n,\square}_{\bar{\rho}_v}: \ & \mathcal{C}_{\Ocal}\rightarrow \mathfrak{Set } \\
    & A \mapsto \{\textrm{Steinberg liftings of } \bar{\rho}_v \textrm{ to } A\}
\end{align*}
This functor is pro-representable by the complete Noetherian local ring $R^{\square,\st}_v:=\Ocal^{\wedge}_{X_{\textrm{St}},\bar{\rho}}$.
 We notice that when we view $X_{\textrm{St}}$ as a scheme over $L$, Theorem \ref{theorem-smooth} tells us, since $q$ is not a root of unity in $L$, and is therefore considerate towards $G_L$, that any localisation of $R^{\square,\st}_v[1/l]$ is a regular ring and thus that $R^{\square,\st}_v[1/l]$ is regular. 

We recall the definition of $\tilde{r}$-discrete series found in Section 2.4.5 in \cite{CHT08}.

\begin{definition}
Let $\tilde{r}_v:G_{F,{\tilde{v}}}\rightarrow \GL_d(\Ocal)$ be a representation such that:
\begin{enumerate}
    \item $\tilde{r}_v\otimes \F$ is absolutely irreducible ($\F$ the residue field of $\Ocal$);
    \item Every irreducible subquotient of $(\tilde{r}_v\otimes \F)|_{I_{\tilde{v}}}$ is absolutely irreducible;
    \item $\tilde{r}\otimes \F \not \cong \tilde{r}\otimes \F(i)$ for each $i=0,...,m$ where $\_(i)$ denotes the twist by the unramified character sending $\Frob\mapsto q^i$. 
\end{enumerate}
Whenever $R$ is an $\Ocal$ algebra, we say a representation $\rho:G_{F,\tilde{v}}\rightarrow \GL_{md}(R)$ is\newline $\tilde{r}$-discrete series if there is a decreasing filtration $\{\Fil^i\}$ of $\rho$ by $R$-direct summands such that 
\begin{enumerate}
    \item $\gr^i\rho \cong \gr^0\rho (i)$ for $i=0,...,m-1$;
    \item $\gr^0\rho|_{I,\tilde{v}}\cong \tilde{r}|_{I,\tilde{v}}\otimes_{\Ocal}R$.
\end{enumerate} 
\end{definition}

\begin{prop}{\label{prop-stacks}}

Suppose $l>h_G$. Let $\tilde{r}$ be a rank $d$ representation as above, and let $n$ be an integer with $d|n$. Let $X_{\tilde{r},n}$ be the moduli space of framed
$\tilde{r}$-discrete series representations of rank $n$, defined over $\Ocal$. Then the base change $(X_{\tilde{r},n})_L$ is smooth over $L$.
\end{prop}
\begin{proof}
Let $S_{\tilde{r}}$ be the moduli stack over $\Ocal$ of $n$-dimensional $\tilde{r}$-discrete representations, so that $S_{\tilde{r}}\cong [X_{\tilde{r}}/\GL_n]$ and let $S_{\mathds{1}}$ be the stack of $m:=n/d$-dimensional $\mathds{1}$-discrete series representations. Let $S^{\textrm{WD}}_{\tilde{r}}$ be the stack over $L$ whose groupoid over $R$ consists of objects $(\rho',N)$ where $\rho'$ is a rank $n=dm$ $\tilde{r}$-discrete series representation with open kernel, and $N$ is an element of $\mathrm{End}_{R}(R^n)$ such that $\rho' N \rho'^{-1}=q^{\nu} N$. Define $S^{\textrm{WD}}_{\mathds{1}}$ analogously. 
Let $t_l$ be the homomorphism $t_l:I\rightarrow \Z_l$ sending any lift of the topological generator of tame inertia to $1\in\Z_l$. 
Recall that there is a morphism $S^{\textrm{WD}}_{\tilde{r}}\rightarrow S_{\tilde{r}}$ given by $(\rho',N)$ is sent to the unique representation $\rho$ given by $\rho(g)= \rho'(g)\exp(t_l(g)N)$ for $g\in I$ and $\rho(\Frob)=\rho'(\Frob)$. 
Recall that this is an isomorphism on the base change to $L$.

Then we have an morphism of algebraic stacks 
$S^{\textrm{WD}}_{\mathds{1}}\rightarrow S^{\textrm{WD}}_{\tilde{r}}$ given by the morphism $(\rho',N) \mapsto (\rho',N) \otimes \tilde{r}$.
We claim that this is an isomorphism. By an exercise in Clifford theory and by assumptions on $\tilde{r}$, the restriction $\tilde{r}|_I$ can be written as a direct sum of pairwise non-isomorphic 
absolutely irreducible $I$-representations $\tau\oplus \tau^{\Frob}\oplus...\oplus \tau^{\Frob^{k-1}}$ for some $k\in \N$. As $\rho'$ is $\mathds{1}$-discrete series in characteristic zero, we see that $(\rho'\otimes \tilde{r})|_I\cong m(\tau\oplus \tau^{\Frob}\oplus...\oplus \tau^{\Frob^{k-1}})$.
Let $V_{\tilde{r}}(R)=\End_{R[I]}(\tilde{r}^m)$ be the space of $I$-equivariant maps of any representation in $S^{\textrm{WD}_{\tilde{r}}}(R)$, and define $V_{\mathds{1}}(R)=\End_{R[I]}(\mathds{1}^m)$ similarly. Note that the map
\begin{align}
    V_{\mathds{1}}(R) &\rightarrow V_{\tilde{r}}(R) \\
    N & \mapsto N\otimes \id_{\tilde{r}}
\end{align}
is injective, and hence is isomorphic onto its image. 
We claim that if $(\rho,N)\in S^{\textrm{WD}}_{\tilde{r}}(R)$, then $N$ is in the image of this map. 

First, note that $N$ is $I$-equivariant. We calculate using Schur's lemma that $V_{\tilde{r}}(R)\cong M_m(R)^k$, since each $\tau^{\Frob^i}$ is absolutely irreducible, and we see the above map corresponds to the diagonal map 
$\Delta: M_m(R)\rightarrow M_m(R)^k$.

The space $V_{\tilde{r}}(R)$ has a natural action of Frobenius on it, and under this action
 $N=(N_1,...,N_k)\in M_m(R)^k$ has $\Frob.(N_1,...,N_k)=q(N_1,...,N_k)$. 
Notice that $\Frob$ induces an isomorphism of the underlying spaces $\tau^m\rightarrow (\tau^{\Frob})^m$, which gives us a commutative diagram
\[\begin{tikzcd}
\tau^m\arrow{r}{\Frob} \arrow{d}{N_1} & (\tau^{\Frob})^m \arrow{d}{qN_2} \\
\tau^m \arrow{r}{\Frob} & (\tau^{\Frob})^m
\end{tikzcd}\]
Hence, we see $(qN_2,...,qN_k,qN_1)=q(N_1,...,N_{k-1},N_k)$, and thus $N$ lies in the image of the diagonal map. This proves the claim.

Let $\chi_{\tilde{r}}=\Hom_I(\tau,\tilde{r})$. Notice that this is an unramified character. 
We claim that $(\Hom_I(\tau,\_)\otimes{\chi_{\tilde{r}}^{-1}},\Delta^{-1}):S^{\textrm{WD}}_{\tilde{r}}\rightarrow S^{\textrm{WD}}_{\mathds{1}}$ is an inverse defining the equivalence. 

We first show that the composition $S^{\textrm{WD}}_{\tilde{r}}(R)\rightarrow S^{\textrm{WD}}_{\mathds{1}}(R) \rightarrow S^{\textrm{WD}}_{\tilde{r}}(R)$ is the identity.
For $(\Theta,N)\in S^{\textrm{WD}}_{\tilde{r}}(R)$, the previous claim gives us an isomorphism on the $N$-part of the stacks $S^{\textrm{WD}}_{\tilde{r}}(R)$, so we focus on the representation part. Since $I$ acts through a finite quotient, and $R$ is an algebra over a characteristic $0$-field, we see that $\Theta$ is semisimple and hence we get a sequence of $I$-representations isomorphisms:
\begin{align*}
    \Theta&\cong \bigoplus_{i=0}^{k-1} \Hom_I(\tau^{\Frob^i},\Theta)\otimes{\chi_{\tilde{r}}^{-1}}\otimes \tau^{\Frob^i} \\
    & \cong \Hom_I(\tau,\Theta)\otimes{\chi_{\tilde{r}}^{-1}}\otimes \bigoplus_{i=0}^{k-1} \tau^{\Frob^i}\\
    &\cong \Hom_I(\tau,\Theta)\otimes{\chi_{\tilde{r}}^{-1}}\otimes \tilde{r}
\end{align*}
To show this isomorphism also respects the $W_F$-action, we observe that each graded part of $\Theta$ has $\gr^i(\Theta)\cong \tilde{r}\otimes \chi(i)$ where $\chi$ is some unramified character. Then we obtain
$$\Hom_{L[I]}(\tau, \gr^i(\Theta))\cong \Hom_{L[I]}(\tau,\tilde{r}\otimes \chi(i))\cong \tilde{r}\otimes \chi(i)$$
so that both sides of the isomorphism are naturally $\tilde{r}\otimes \chi(i)$ as $W_F$-representations.
Hence, the composition $S^{\textrm{WD}}_{\tilde{r}}(R)\rightarrow S^{\textrm{WD}}_{\mathds{1}}(R) \rightarrow S^{\textrm{WD}}_{\tilde{r}}(R)$ is the identity. 

To show $S^{\textrm{WD}}_{\mathds{1}}(R)\rightarrow S^{\textrm{WD}}_{\tilde{r}}(R) \rightarrow S^{\textrm{WD}}_{\mathds{1}}(R)$ is the identity, let $\rho\in S_{\mathds{1}}(R)$. Then the natural map
\begin{align}
    \rho &\rightarrow \Hom_I(\tau,\rho \otimes \tilde{r})\\
    v &\mapsto \{w\mapsto v\otimes w\}
\end{align}
 defines an $I$ isomorphism. So we need only check that $\rho\otimes\chi_{\tilde{r}}$ and $\Hom_I(\tau,\rho \otimes \tilde{r})$ have the same action of Frobenius. This can be checked again, by looking at the character $\gr^i(\rho)$.
 Hence, we have exhibited an equivalence of categories $S_{\mathds{1}}\leftrightarrow S_{\tilde{r}}$.

Given a choice of Frobenius $\Frob$ and a topological generator of the tame inertia group $s$ we can explicitly write an isomorphism of stacks 
\begin{align*}
    S_{\mathds{1}} &\cong [X_{\textrm{St}}/\GL_m] \\
    \rho & \mapsto \Big(\rho(\Frob), \log\big(\rho(s)\big)\Big)\\
    \rho_{\Phi}(\Frob^n x)=\Phi^n\exp(Nt_l(x)) & \leftmapsto (\Phi,N)
\end{align*}
As $(X_{\textrm{St}})_L$ is a smooth scheme by Theorem \ref{theorem-smooth}, it shows that $S_{\mathds{1}}[1/l]$ is a smooth stack, and thus that $S_{\tilde{r}}[1/l]$  and $(X_{\tilde{r},n})_L$ are smooth. 
\end{proof}

In light of this proposition, if $\bar{\rho}:G_{F,\tilde{v}}\rightarrow \GL_n(\F)$ is $\tilde{r}$-discrete series, we let $R^{\square,\tilde{r}}_v$ be the universal lifting ring of $\tilde{r}$-discrete series representations. By the proposition, the ring $R^{\square,\tilde{r}}_v[1/l]$ is regular at every maximal ideal.

\subsubsection{Deformation rings at primes above $l$}

For $v\in S_l$, let $\bar{I}_{\tilde{v}}$ be the inertia subgroup of $G_{F,\tilde{v}}^{\textrm{ab}}$, let $\bar{I}_{\tilde{v}}(l)$ be the pro-$l$ part, and let $\Lambda_{\tilde{v}}:=\Ocal[[\bar{I}_{\tilde{v}}(l)^n]]$, which we can identify with the universal lifting algebra of an ordered set of inertial characters $\{\bar{\chi}_i:I_{\tilde{v}}\rightarrow \F^{\times}\}_{i=1,...,n}$. Following chapter 3 of \cite{geraghty_2018} we can define a lifting $\Lambda_{\tilde{v}}$-algebra $R^{\triangle}_{v}$ as follows. 
Take the universal lifting ring $R^{\square,\Lambda}_{v}$, such that a morphism $r:R^{\square,\Lambda}_{v}\rightarrow A$ corresponds to a pair $(\rho,\{\chi_i\}_{i=1,...,n})$ consisting of a representation $\rho:G_v\rightarrow \GL_n(A)$ lifting $\bar{\rho}$ and a sequence of characters $\chi_i:I_{\tilde{v}}\rightarrow A^{\times}$. 
Let $\cat{Flag}$ be the flag variety defined over $\Ocal$. There is a subscheme $\Gc$ of $\cat{Flag}\times_{\Ocal}\Spec R^{\square,\Lambda}_{v}$ whose $A$-points are the triples $(\Fil,\rho, \{\chi_i\})\in (\cat{Flag}\times_{\Ocal}\Spec R^{\square})(A)$ such that $\rho:G_v\rightarrow \GL_n(A)$ preserves the filtration $\Fil$ on $A^n$, and such that the action of $I_{\tilde{v}}$ on the graded part $\Fil_j/\Fil_{j-1}$ is $\chi_j$. Then $R^{\triangle}_v$ is defined as the image of the natural morphism $R^{\square,\Lambda}_{v}\rightarrow \Gamma(\Gc,\Oc_{\Gc})$.

It is a fact (see Lemma 3.3 of \cite{geraghty_2018}) that the morphism $R^{\square}_v\rightarrow \Ocal$ corresponding to a representation $\rho:G_v\rightarrow \GL_n(\Ocal)$
factors through $R^{\triangle}_v$
if and only if
$\rho$ is $\GL_n(\Ocal)$-conjugate to an upper triangular representation with diagonal characters equal to $\chi_1,...,\chi_n$ when restricted to inertia. 

\begin{definition}
    If $A$ is a $\Z_l$-algebra and $v\in S_l$, we call a representation $\rho:G_v\rightarrow \GL_n(A)$ \emph{ordinary} if it is $\GL_n(A)$-conjugate to an upper triangular matrix. Likewise, if $\rho:\Gal(\bar{F}:F)\rightarrow \GL_n(A)$ is a global Galois representation, we say $\rho$ is ordinary if $\rho|_{G_v}$ is ordinary at all places $v\in S_l$. 
\end{definition}

In this terminology, the fact above can be restated as `a point $x:R^{\square}_v\rightarrow \Ocal$ factors though $R^{\triangle}_v$ if and only if the corresponding representation $\rho_x$ is ordinary'.

\begin{lemma}\label{lemma Sl}

    Suppose that $\bar{\rho}_v:G_{F,\tilde{v}}\rightarrow \GL_n(\F)$ is an ordinary Galois representation with diagonal characters $\bar{\chi}_1,\bar{\chi}_2,...,\bar{\chi}_n$, such that no pair $i<j$ has $\chi_i=\varepsilon \chi_j$ where $\varepsilon$ is the cyclotomic character. Then $R^{\triangle}_v[1/l]$ is formally smooth of dimension $[F_{\tilde{v}}:\Q_l]\frac{n(n+1)}{2}+n^2$ over $L$.
\end{lemma}
\begin{proof} 
This follows from Lemmas 3.17 and 3.7 of \cite{geraghty_2018}. (To apply Lemma 3.17 as stated in \cite{geraghty_2018}, one must note that $\Gc^{ar}$ is a union of irreducible components of $\Gc$ and $\bar{\rho}_v$ lies in the open subset of $\Gc[1/l]$ whose closure is defined to be $\Gc^{ar}$. Thus $R^{\triangle,ar}_v=R^{\triangle}_v$.) 
\end{proof}

\subsection{Local-Global compatibility}

We start by introducing the group $\mathcal{G}_n$ from \cite{CHT08}, defined as the group scheme that is the semi-direct product of $\GL_n\times \GL_1$ with $C_2=\{1,j\}$ where $j$ acts as 
$$j(g,\mu)j^{-1}=\big(\mu (g^{-1})^\mathrm{T},\mu\big).$$
By Lemma 2.1.1 of \cite{CHT08}, representations $r:G_{F^+}\rightarrow \mathcal{G}_n(R)$ such that $r^{-1}(\GL_n(R)\times \GL_1(R))=G_F$ correspond with pairs $(\rho,\chi)$ where $\rho$ is an $n$-dimensional representation of $G_F$ and $\chi$ is a character of $G_{F^+}$ such that $\rho^c\cong \chi\rho^{\vee}$, and $c\in G_{F^+}$ is sent to $j$.

For brevity, whenever we have a homomorphism $r:G_{F^+}\rightarrow \mathcal{G}_n(R)$ and a subgroup $H\subset G_{F^+}$, we use $r|_{H}$ to mean restriction to $H$, followed by projection to $\GL_n$. 
Typically, $H$ will be the subgroup $G_F$ or its localisations $G_{F,w}$.

\begin{prop}\label{Prop-reps}
Suppose that $\mathfrak{m}\trianglelefteq \T^{T,\ord}(U(l^{\infty}),\Ocal)$ is a maximal ideal with residue field $\F$. Then there is a unique continuous semisimple representation 
$$\bar{r}_{\mathfrak{m}}:G_F\rightarrow \GL_n(\F)$$ 
such that:
\begin{enumerate}
    \item $$\bar{r}_{\mathfrak{m}}^c\cong \bar{r}_{\mathfrak{m}}^{\vee}\otimes \varepsilon^{1-n};$$
    \item $\bar{r}_{\mathfrak{m}}|_w$ is unramified at all places $v$ of $F^+$ outside $T$ ;
    \item If $v$ additionally splits as $v=ww^c$ in $F$, then the characteristic polynomial of $\bar{r}_{\mathfrak{m}}(\Frob_w)$ is 
    $$X^n-T^{(1)}_w X^{n-1}+...+(-1)^j N(w)^{\frac{j(j-1)}{2}}T^{(j)}_w X^{n-j}+...+(-1)^n N(w)^{\frac{n(n-1)}{2}}T^{(n)}_w$$
    modulo $\mathfrak{m}$; 
    \item Let $\tilde{r}_{\tilde{v}}:G_F\rightarrow \GL_{m_v}(\Ocal)$ be as in Section 3.3 (p97) of \cite{CHT08} (note: this is constructed from the smooth representation $\rho_v:G_D(F^+_v)\rightarrow \GL(M_v)$ via the Jacquet-Langlands and local Langlands correspondences). If $v\in S_D$ and $U_v=G_D(\Ocal_{F^+,v})$, then $\bar{r}_{\mathfrak{m}}|_{G_{F,v}}$ is $\tilde{r}_{\tilde{v}}$-discrete series.
\end{enumerate}

\end{prop}
\begin{proof}
Apart from statement 4, this is Proposition 2.28 in \cite{geraghty_2018}, so we prove only this part. By the argument of Proposition 2.28 in \cite{geraghty_2018}, the maximal ideals of $\T$ are in bijection with those of $\T/m_{\Lambda}$. Hence, this proposition follows immediately from the classical situation (that is, usual automorphic forms for $G_D$ rather than Hida families of ordinary automorphic forms). The proof of this can be found in Proposition 3.4.2 of \cite{CHT08}, which proves the proposition.
\end{proof}

\begin{prop}\label{LGC2}
If $\mathfrak{m}$ is non-Eisenstein (that is $\bar{r}_{\mathfrak{m}}$ is irreducible), then $\bar{r}_{\mathfrak{m}}$ can be extended to a representation $\bar{r}_{\mathfrak{m}}:G_{F^+}\rightarrow \mathcal{G}_n(\F)$, and this representation can be lifted to a representation 
$$r_{\mathfrak{m}}:G_{F^+}\rightarrow \mathcal{G}_n\Big(\T^{T,\ord}(U(l^{\infty}),\Ocal)_{\mathfrak{m}}\Big)$$ 
such that:

\begin{enumerate}
    \item If $\nu:\mathcal{G}_n\rightarrow \GL_1$ is the second projection, then $\nu\circ r_{\mathfrak{m}}=\varepsilon^{1-n}\delta^{\mu_{\mathfrak{m}}}_{F/F^+}$ where $\varepsilon$ is the cyclotomic character, $\delta_{F/F^+}$ is the non-trivial character of $G_{F^+}/G_F$, and $\mu_m\in\Z/2$;
    \item $\bar{r}_{\mathfrak{m}}|_{\tilde{v}}$ is unramified at all places $v\notin T$;
    \item If $v$ in addition splits as $v=ww^c$ in $F$, then the characteristic polynomial of $\bar{r}_{\mathfrak{m}}(\Frob_w)$ is 
    $$X^n-T^{(1)}_w X^{n-1}+...+(-1)^j N(w)^{\frac{j(j-1)}{2}}T^{(j)}_w X^{n-j}+...+(-1)^n N(w)^{\frac{n(n-1)}{2}}T^{(n)}_w;$$
    \item If $v\in S_D$, then $r_{\mathfrak{m}}|_{G_{F,\tilde{v}}}$ is $\tilde{r}_{\tilde{v}}$-discrete series.
   
\end{enumerate}

\end{prop}
\begin{proof}
    As with the previous proposition, this is Proposition 2.29 in \cite{geraghty_2018} along with the additional statement 4, so we prove only this final statement.
    By the proof of Proposition 2.29 of \cite{geraghty_2018}, we may find a sequence of maximal ideals $\mathfrak{m}_b\subset \mathbb{T}^{T,\ord}(U(l^{b}),\Ocal)$ such that $\T_{\mathfrak{m}}=\varprojlim_b\mathbb{T}^{T,\ord}(U(l^{b}),\Ocal)_{\mathfrak{m}_b}$, and we define $r_{\mathfrak{m}}=\varprojlim_b r_{\mathfrak{m}_b}$. By Lemma 3.4.4 of \cite{CHT08}, each $r_{\mathfrak{m}_b}|_{G_{F,\tilde{v}}}$ is $\tilde{r}_{\tilde{v}}$-discrete series, and so now it remains to show that $r_m|_{G_{F,\tilde{v}}}$ is, too.
    Since 
    $$r_{\mathfrak{m}_b}\otimes \mathbb{T}^{T,\ord}(U(l^{c,c}),\Ocal)_{\mathfrak{m}_c}=r_{\mathfrak{m}_c}$$ whenever $b>c$,, it follows that the filtration $\Fil^i_b$ on $r_{\mathfrak{m_b}}$ descends to a filtration $\Fil^i_b\otimes \mathbb{T}^{T,\ord}(U(l^{c,c}),\Ocal)_{\mathfrak{m}_c}$ on 
    $r_{\mathfrak{m_c}}$, and that the graded parts have 
    $$[\gr^i(r_{\mathfrak{m},b})]\otimes \mathbb{T}^{T,\ord}(U(l^{c,c}),\Ocal)_{\mathfrak{m}_c}\cong \gr^i[r_{\mathfrak{m},b}\otimes\mathbb{T}^{T,\ord}(U(l^{c,c}),\Ocal)_{\mathfrak{m}_c}].$$
    It follows that $\Fil^i_b\otimes \mathbb{T}^{T,\ord}(U(l^{c,c}),\Ocal)_{\mathfrak{m}_c}$ is a defining filtration on $r_{\mathfrak{m}_c}$. 
    From Lemma 2.4.25 of \cite{CHT08}, such a filtration is unique, so we have a compatible system of filtrations on the $r_{\mathfrak{m}_b}$ which lift to a filtration on $r_{\mathfrak{m}}|_{G_{F,\tilde{v}}}$. We see from this compatibility that ${\gr^i(r_{\mathfrak{m}})}=\varprojlim_b {\gr^i(r_{\mathfrak{m}_b})}$, and so $r_{\mathfrak{m}}|_{G_{F,\tilde{v}}}$ is $\tilde{r}_{\tilde{v}}$-discrete series.
    
\end{proof}

To complete the results we need for local-global compatibility, we need the following lemma:

\begin{lemma}
Let $\tilde{v}\in \tilde{S_l}$, and let $R^{\triangle}_{\tilde{v}} $ be as before. Then there is a map $R^{\triangle}_{\tilde{v}}\rightarrow \T^{T,\ord}(U(l^{\infty}),\Ocal)_{\mathfrak{m}}$ such that 
\[\begin{tikzcd}
G_{F,{\tilde{v}}} \arrow{r}{\rho^{\triangle}} \arrow{rd}{r_{\mathfrak{m}}}  & \Gc_n(R^{\triangle}_{\tilde{v}}) \arrow{d} \\
 & \Gc_n\Big(\T^{T,\ord}(U(l^{\infty}))\Big)
\end{tikzcd}\]
    commutes.
\end{lemma}
\begin{proof}
    This follows directly from Corollary 4.3 of \cite{geraghty_2018}.
\end{proof}

\subsection{Global deformation rings}

Let $F/F^+$ and let $\bar{\rho}:G_{F}\rightarrow \GL_n(\F)$ be a representation with local representations $\rho_w=\bar{\rho}|_{G_{F,w}}$, where $w$ is a place of $F$. Let $R$ be the set of places $v$ of $F^+$ such that $v$ splits and there is a place $w$ of $F$ above $v$ where $\rho$ ramifies. Set $T=S_l\coprod S_D\coprod R$, and define $\tilde{T}$ as before.
\\\\
Assume that: 
\begin{itemize}
    \item the representation $\bar{\rho}$ is an irreducible automorphic representation. That is, there is a non-Eisenstein maximal ideal $\mathfrak{m}\trianglelefteq \T^{T,\ord}(U(l^{\infty}),\Ocal))$ so that $\bar{\rho}\cong \bar{r}_{\mathfrak{m}}$; 
    \item the subgroup $\rho(G_{F^+(\zeta_l)})\subseteq \mathcal{G}_n(\F)$ is adequate in the sense of Definition 2.3 of \cite{thorne};
    \item the representation $\bar{\rho}$ is unramified outside $\Tilde{T}$;
    \item at any place $v\in R$, any lift of $\bar{\rho_v}$ to $\overline{\Q_l}$ is non-degenerate in the sense of Section 3.3 of \cite{Jackshottonduke}, in particular they lie on a single irreducible component of $\Loc^{\square}_{\GL_n,\Q_l}$;
    \item For each $v\in S_l$, have $\Hom_{G_{F,\tilde{v}}}(\bar{\rho}_{\tilde{v}},\bar{\rho}_{\tilde{v}}\varepsilon)=0$ for $\varepsilon$ the cyclotomic character. 
\end{itemize}
As $\bar{\rho}\cong \bar{r}_{\mathfrak{m}}$ is irreducible, it can be extended to a representation $\bar{\rho}:G_{F^+}\rightarrow \mathcal{G}_n(\F)$ such that $\nu\circ\bar{\rho}=\varepsilon^{1-n}\delta_{F/F^+}^{\mu_\mathfrak{m}}$ via Proposition \ref{LGC2}. We fix such an extension.

For each $v\in T$, define $R^{\square}_v$ as the framed deformation ring for $\bar{\rho}_{\tilde{v}}$.
Set
$$R^{\loc}:=\left(\bighotimes_{\Ocal, v\in S_l}R^{\triangle}_{v}\right)\hotimes_{\Ocal} \left(\bighotimes_{\Ocal,v\in S_D}R^{\square,{\tilde{r}_{\tilde{v}}}}_{v}\right)\hotimes_{\Ocal} \left(\bighotimes_{\Ocal,v\in R}R^{\square}_{v}\right)$$
as the local deformation ring for $\bar{\rho}$. 
Our first observation is that, since each $R^{\triangle}_{v}$ is a $\Lambda_{\tilde{v}}$-module, the algebra $R^{\loc}$ inherits the structure of a $\bighotimes_{v\in S_l}\Lambda_{\tilde{v}} \cong \Lambda$-module. 
The isomorphism $\bighotimes_{v\in S_l}\Lambda_{\tilde{v}} \cong \Lambda$ is inherited from the group isomorphisms 
$$T_n(\mathfrak{l})\cong \prod_{v\in S_l}T_n\Ocal_{F^+,v}(l)\cong \prod_{v\in S_l}T_n\Ocal_{F,\tilde{v}}(l) \cong \prod_{v\in S_l} \bar{I}_{\tilde{v}}(l)^n$$
where the final isomorphism is given by the Artin map of local class field theory.

\begin{lemma}{\label{lem: Rloc}}
    The ring $R^{\loc}[1/l]$ is regular.  
\end{lemma}
\begin{proof}
Recall the construction of the complete ring $R^{\triangle}_v$ as the image of $R^{\square,\Lambda}_{v}=R^{\square}_v\hotimes_{\Oc}\Lambda_v$ in the global sections of $\Gc\subseteq \cat{Flag}\times_{\Oc}R^{\square,\Lambda}_{v}$. We `de-complete' $R^{\triangle}_v$ as follows:
Observe that $R^{\square}_v$ is the completion of a local ring $\tilde{R}^{\square}_v$ at a closed point on a finite-type scheme over $\Oc$. The ring $\Lambda_v=\Oc[[\bar{I}_{\tilde{v}}(l)^n]]$ is a completed group algebra of a group which is topologically finitely generated, generated by a fixed choice of generators $\{s_i\}$. 
So we can choose a subring $\tilde{\Lambda}_v=\Ocal[s_i^{\pm 1}]/\langle\textrm{relations}\rangle\subseteq \Lambda_v$ of finite type over $\Oc$ which is dense in $\Lambda_v$. Thus, there is a ring $\tilde{R}^{\square,\Lambda}_{v}$ of finite type over $\Oc$ whose completion is $R^{\square,\Lambda}_{v}$.

We can define a closed subscheme $\tilde{\Gc}\subseteq \cat{Flag}\times_{\Oc}\Spec(\tilde{R}^{\square,\Lambda}_{v})$ cut out by the same equations for $\Gc$ as in the definition of $R^{\triangle}_v$. Of course, there is a natural commutative diagram
\[\begin{tikzcd}
	{\tilde{R}^{\square,\Lambda}_{v}} & {\Oc(\tilde{\Gc})} \\
	{R^{\square,\Lambda}_{v}} & {\Oc(\Gc)}
	\arrow["{\tilde{\phi}}", from=1-1, to=1-2]
	\arrow[hook, from=1-1, to=2-1]
	\arrow[from=1-2, to=2-2]
	\arrow["\phi"', from=2-1, to=2-2]
\end{tikzcd}\]
We set $\tilde{R}^{\triangle}_v$ as the image of $\tilde{\phi}$. It is a finite-type ring over $\Oc$ and, since the equations defining $\Gc\subseteq \cat{Flag}\times\Spec(R^{\square,\Lambda}_v)$ are rational (that is, the defining ideal $\Ic=\tilde{\Ic}\Oc_{\cat{Flag}\times\Spec(R^{\square,\Lambda}_v)}$ for some ideal $\tilde{\Ic}\subseteq \cat{Flag}\times\Spec(\tilde{R}^{\square,\Lambda}_v)$), the image $\im(\phi)=R^{\triangle}_v$ is a completion of $\tilde{R}^{\triangle}_v$.

It follows that each of $R^{\square}_v$, $R^{\square,\tilde{r}_{\tilde{v}}}_v$ and $R^{\triangle}_v$ (for $v\in R, S_D,S_l$ respectively) are completions of local rings at a closed point $P_v$ on a finite-type $\Oc$-scheme $X_v$. By the last two hypotheses on $\bar{\rho}$ listed above, Lemma \ref{lemma Sl} and Proposition 3.6 of \cite{Jackshottonduke}, we see that $R^{\triangle}_v[1/l]$ (for $v\in S_l$) is regular and $R^{\square}_v[1/l]$ (for $v\in R$) is formally smooth. Thus, the closed points $P_v$ on $X_v$ (where $v\in S_l\cup R$) lie on an open subscheme $U_v\subseteq X_v$ whose generic fibre $U_v[1/l]$ is smooth over $L$. By Proposition \ref{prop-stacks}, the same is true for $R^{\square,\tilde{r}_{\tilde{v}}}_v$ for $v\in S_D$. 
Set
\[\tilde{R}^{\loc}:=\left(\bigotimes_{\Ocal, v\in S_l}\tilde{R}^{\triangle}_{v}\right)\otimes_{\Ocal} \left(\bigotimes_{\Ocal,v\in S_D}\tilde{R}^{\square,{\tilde{r}_{\tilde{v}}}}_{v}\right)\otimes_{\Ocal} \left(\bigotimes_{\Ocal,v\in R}\tilde{R}^{\square}_{v}\right)\]
Then $\tilde{R}^{\loc}$ is of finite type over $\Oc$ and has a maximal ideal $m$, corresponding to the closed point $(P_v)_{v}$, with respect to which $R^{\loc}$ is the $m$-adic completion.
In addition, $\tilde{R}^{\loc}[1/l]$ is a regular $L$-algebra. To show that $R^{\loc}[1/l]$ is a regular ring is now a simple application of Lemma \ref{lemma-completion} and \cite[\href{https://stacks.math.columbia.edu/tag/07NY}{Lemma 07NY}]{stacks-project}.
\end{proof}

In fact, the same argument shows that $R_{\infty}[1/l]$ is a regular ring whenever $R_{\infty}$ is a power series ring in a finite number of variables with coefficients in $R^{\loc}$.

Let $\mathcal{S}$ be the following tuple
\[\mathcal{S}=\big(F/F^+, T, \tilde{T}, \varepsilon^{1-n}\delta^{\mu_m}_{F/F^+} ,\{R^{\triangle,ar}_{v}:v\in S_l\},\{R^{\square,st}_v:v\in S_D\},\{R^{\square}_v:v\in R\}\big)\] 
and say that $\rho:G_{F^+}\rightarrow \mathcal{G}(A)$ is a lifting of $\bar{\rho}$ to $A\in \mathcal{C}_{\Lambda}$ of type $\mathcal{S}$ if:

\begin{enumerate}
    \item $\rho|_{G_F}$ lifts $\bar{r}_m$;
    \item $\rho$ is unramified outside $T$;
    \item For $v\in S_D$, the local representation $\rho_v$ is $\tilde{r}$-discrete series and gives rise to the morphism $R^{\square}_v\rightarrow A$ which factors through $R^{\square,\tilde{r}}_v$;
    \item For $v\in S_l$, the restriction $\rho_v$ and the $\Lambda$-structure on $A$ give a morphism $R^{\square}_{v}\otimes \Lambda\rightarrow A$ which factors through $R^{\triangle}_{v}$;
    \item $\nu\circ \rho=\varepsilon^{1-n}\delta^{\mu_m}_{F/F^+}$.
\end{enumerate}
By Proposition 2.2.9 of \cite{CHT08}, we can construct the universal deformation ring $R_{\mathcal{S}}^{\univ}$ and the universal lifting ring $R_{\mathcal{S}}^{\square}$. 

Let $h_0=[F^+:\Q]\frac{n(n-1)}{2}+[F^+:\Q]\frac{n(1-(-1)^{\mu_{\mathfrak{m}-1}})}{2}$, and let $h$ be an integer larger than both $h_0$ and $\dim[H^1_{\mathcal{L}^{\bot}}(G_{F^+,T},\ad\bar{\rho}(1))]$. (Here, the space $H^1_{\mathcal{L}^{\bot}}(G_{F^+,T},\ad\bar{\rho}(1))$ is a particular subspace of the cohomology group $H^1(G_{F^+,T},\ad\bar{\rho}(1))$ of the Galois group $G_{F^+,T}$ of the maximal extension of $F^+$ unramified outside of $T$, defined in Proposition 4.4 of \cite{thorne}.)

After Thorne \cite{thorne}, we will call a triple $(Q,\tilde{Q},\{\bar{\psi}_v\}_{v\in Q})$ a Taylor-Wiles triple if:

\begin{enumerate}
    \item $Q$ is a set of primes of $F^+$ which split in $F$;
    \item $l|\textrm{Nm}_{F^+}(v)-1$ for each $v\in Q$;
    \item $|Q|=h$;
    \item $\tilde{Q}$ is the set $\{\tilde{v}|v\in Q\}$;
    \item For each $v\in Q$, the representation $\bar{\rho}|_{G_v}$ splits as a direct sum into $\bar{s}_v\oplus \bar{\psi}_v$ where $\bar{\psi}_v$ is a generalised eigenspace with eigenvalue $\bar{\alpha}_{v}\in\mathbb{F}$ of dimension $d_v$.
\end{enumerate}

For any Taylor-Wiles set $Q$ we can define a deformation problem $\mathcal{S}(Q)$, which is the same as $\mathcal{S}$, but in addition, we now allow $\rho_{\tilde{v}}$ for $v\in Q$ to ramify in the following way:
$\rho_{\tilde{v}}$ splits as a direct sum $s\oplus \psi$, which lift $\bar{s}$ and $\bar{\psi}$ respectively, such that $s$ is unramified, and $\psi|_{I_v}:I_v\rightarrow \GL_{d_v}$ factors through the scalar action on the underlying representation space. 
Using Proposition 2.2.9 in \cite{CHT08} again, we can now take the universal deformation ring $R^{\univ}_{\mathcal{S}(Q)}$. 
Because stipulating that the local deformations at Taylor-Wiles primes are unramified is a closed condition, this presents us with a surjection $R^{\univ}_{\mathcal{S}(Q)}\twoheadrightarrow R_{\mathcal{S}}^{\univ}$.
Further, we also have a natural map $R^{\loc}\rightarrow R^{\univ}_{\mathcal{S}(Q)}$ given by restrictions to the local subgroups at the level of functors.

\begin{prop}
    For each $N\in \mathbb{N}$, we can find a Taylor-Wiles triple $(Q_N,\tilde{Q}_N,\{\bar{\psi}_v\}_{v\in Q})$ such that $l^N||\Nm(v)-1 $ for all $v\in Q_N$ and the global deformation ring $R^{\univ}_{\mathcal{S}(Q)}$ can be topologically generated over $R^\loc$ by $h-h_0$ generators.
\end{prop}
\begin{proof}
    This follows from Lemma 4.4 of \cite{thorne} applied in the case of Theorem 8.6.
\end{proof}
In light of this proposition, set $R_{\infty}=R^{\loc}[[X_1,...,X_h]]$, set $R_N=R^{\univ}_{\mathcal{S}(Q_N)}$, and set $R_0=R^{\univ}_{\mathcal{S}}$ so that we have surjections $R_{\infty}\twoheadrightarrow R_N$ and $R_N\twoheadrightarrow R_0$.

We now define some important subgroups of $G_D(\A)$.
\begin{definition}
For $v\in Q_N$, suppose that $\bar{r}|_v=\bar{s}\oplus \bar{\psi}$ as before, with $\bar{\psi}$ a $d_v$-dimensional semisimple unramified representation with all Frobenius eigenvalues equal. 
We take the group $U_i(\tilde{v})$ to be the subgroup of $U_{v}\subseteq G_D(F^+_v)$ (identified with $\GL_n(F_{\tilde{v}})$ via the isomorphism $i_{\tilde{v}}$) of elements that take the form
$$\begin{pmatrix} \varpi_{\Tilde{v}} * & * \\ 0 & aI_{d_v} \end{pmatrix}$$
modulo $\tilde{v}$
with $a\equiv1\mod{\tilde{v}}$ when $i=1$, and arbitrary when $i=0$. 
Set $U_i(Q)=U^Q\times \prod_{v\in Q} U_i(\tilde{v})\subseteq G_D(\A)$.
\end{definition}
Let $\Delta_N$ be the maximal $l$-power quotient of $U_0(Q_N)/U_1(Q_N)\cong \prod_{v\in Q_N} k(\tilde{v})^{\times}$.
We may view $\Delta_N$ as the maximal $l$-quotient of $\prod_{v\in Q_N} k(\tilde{v})^{\times}\cong (\Z/l^N)^q$.
We claim there is an action of $\Delta_N$ on the ring $R^{\univ}_{\mathcal{S}(Q)}$. 
The map $\det\circ r^{\univ}_N:I_{F,\tilde{v}}\rightarrow (R^{\univ}_{\mathcal{S}(Q)})^{\times}$ given by the determinant of the universal deformation 
$r^{\univ}_N:=r^{\univ}_{\mathcal{S}(Q_N),\bar{\rho}}$ factors through the kernel of $(R^{\univ}_{\mathcal{S}(Q)})^{\times}\rightarrow \F^{\times}$
which is an abelian $l$-power group. By local class field theory, there is an isomorphism $I_{F^{\textrm{ab}},\tilde{v}}\rightarrow \Ocal_{F,\tilde{v}}^{\times}$, and the 
$l$-power quotient of this group is the $l$-power quotient of $k(\tilde{v})^{\times}$. We hence see that there is a map $\Delta_N\rightarrow (R^{\univ}_{\mathcal{S}(Q_N)})^{\times}$ 
and thus a ring map $\Lambda[\Delta_N]\rightarrow R^{\univ}_{\mathcal{S}(Q)}$, so that $R^{\univ}_{\mathcal{S}(Q_N)}$ inherits the structure of a finitely generated $\Lambda[\Delta_N]$-algebra. Notice that if $a_N$ is the augmentation ideal of $\Lambda[\Delta_N]$, 
then $R^{\univ}_{\mathcal{S}(Q_N)}/a_N$ is the ring of the universal deformation ring which parametrises Galois deformations of type $\mathcal{S}$. (These deformations are required to be unramified at places above $Q_N$.)
Note that $\Delta_N\cong (\Z/l^n\Z)^h$ by our choice of $Q_N$.

As in Chapter 4, we can construct the Hecke operators 
\[\T^{T\cup Q_N,\ord}\big(U_1(Q_N)(l^{\infty}),\Ocal\big)\]
and, through a map $\T^{T\cup Q_N,\ord}\big(U_1(Q_N)(l^{\infty}),\Ocal\big)\rightarrow \T^{T,\ord}(U(l^{\infty}),\Ocal)$, we can lift our choice of maximal ideal $\mathfrak{m}$ to a maximal ideal $\mathfrak{m}_N\subset \T^{T\cup Q_N,\ord}(U_1(Q_N)(l^{\infty}),\Ocal)$. Set $\T_{N,1}:=\T^{T\cup Q_N,\ord}(U_1(Q_N)(l^{\infty}),\Ocal)^{\wedge}_{\mathfrak{m}_N}$ as the $\mf_N$-adic completion.
As in Proposition \ref{LGC2}, we can construct a representation $r_{\mathfrak{m}_N}:G_{F^+}\rightarrow \mathcal{G}_n(\T_{N,1})$ which by the proof of Theorem 6.8 of \cite{thorne} gives us an $\mathcal{S}(Q_N)$-lifting of $\bar{\rho}$. Hence, we get a surjection $R^{\univ}_{\mathcal{S}(Q)}\twoheadrightarrow \T_{N,1}$ for each $N$. 

\subsection{Patching}

We now define a module $H_N$ over $\T^{T\cup Q_N,\ord}(U_1(Q_N)(l^{\infty}),\Ocal)_{m}$ for each set $Q_N$, and quote a patching theorem that will allow us to construct the patched `limit' module $H_{\infty}$, which we use to prove our local freeness result.

Define the space of automorphic forms $S^{\ord}(U_i(Q_N)(l^{\infty}),L/\Ocal)_m$ as before and set $H_0=S^{\ord}(U(l^{\infty}),L/\Ocal)_m^{\vee}$. In Proposition 5.9 of
\cite{thorne}, Thorne describes a projection $\Pr_v$ on $S^{\ord}(U_i(Q_N)(l^{\infty}),L/\Ocal)_m$ and, in Theorem 6.8, modules
$$H_{i,N}:=\prod_{v\in Q_N} \Pr_v [S^{\ord}(U_i(Q_N)(l^{\infty}),L/\Ocal)_m]^{\vee}$$ 
with the following properties:

\begin{prop}{\cite{thorne}}
\begin{enumerate}
    \item $H_{1,Q_N}$ is a free $\Lambda[\Delta_{Q_N}]$-module and restriction to $S^{\ord}(U_0(Q_N)(l^{\infty}),L/\Ocal)_m$ gives an isomorphism $H_{1,Q_N}/a_N\cong H_{0,Q_N}$.
    \item The map 
    $$(\prod_{v\in Q_N} \Pr_{\tilde{v}})^{\vee}:H_{0,Q_N}\rightarrow H_0$$ 
    is an isomorphism.
\end{enumerate}
\end{prop}

\begin{proof}
    This is a subclaim of Theorem 6.8 in \cite{thorne}.
\end{proof}

\begin{theorem}[Patching]

Let $R\twoheadrightarrow \T $ be a surjective $\Lambda$-algebra homomorphism with $\T$ a finite $\Lambda$-algebra. 
    Define $S_N=\Lambda [(\Z/l^n\Z)^h]\cong \Lambda[\Delta_{Q_N}]$ with augmentation ideal $\mathfrak{a}_N$ and define the inverse limit $ S'_{\infty}:= \varprojlim \Lambda[\Delta_{Q_N}]\cong \Lambda[[Y_1,...,Y_h]]$. 
    Set $S_{\infty}=S'_{\infty}\hotimes_{\Ocal}\mathcal{T}$, where $\mathcal{T}=\Ocal[[X_1,...,X_{|T|n^2}]]$.
Suppose we have the following data:
\begin{enumerate}
    \item Integers $t,h\geq 1$; 
    \item A finite $\T$-module H;

    \item For each $N\geq 1$ have 
    \begin{enumerate}
        \item $R_N\twoheadrightarrow \T_N$ are $S_N$-algebra homomorphisms, such that reduction modulo $\mathfrak{a}_N$ reduces the map to $R\twoheadrightarrow \T$; 
        \item a finite $\T_N$-module $H_N$, which is finite and free over $S_N$ and whose $S_N$-rank is independent of $N$; 
    \end{enumerate}
    \item An $S_{\infty}$-algebra $R_{\infty}$ such that $R_{\infty}\twoheadrightarrow R_N$ with kernel $\ker(S_{\infty}\rightarrow S_N)R_{\infty}$.
\end{enumerate}

Then there is an $R_{\infty}\otimes S_{\infty}$-module $H_{\infty}$, such that 
\begin{enumerate}
    \item $H_{\infty}/aH_{\infty}\cong H$;
    \item $H_{\infty}$ is a finite free $S_{\infty}$-module; 
    \item The action of $S_{\infty}$ on $H_{\infty}$ factors through that of $R_{\infty}$. 
\end{enumerate}
\end{theorem}
\begin{proof}
The details of the Taylor-Wiles-Kisin patching method used here are no different to chapter 4.3 of \cite{geraghty_2018}. One can also find details in chapter 8 of \cite{thorne} under the heading `another patching argument'.
\end{proof}

\begin{theorem}\label{theorem locfree}
The module $H_0[1/l]$ is a finite locally free $R^{univ}_{\mathcal{S}}[1/l]$-module.
\end{theorem}

\begin{proof}
We calculate that 
\begin{align*}
    \dim(S_{\infty})&=\dim(\Lambda)+h+|T|n^2\\
    &=n[F^+:\Q]n+h+|T|n^2,
\end{align*}
and that 
\begin{align*}
    \dim(R_{\infty})&=1+\sum_{v\in S_l}\left([F_{\tilde{v}}:\Q_l]\frac{n(n+1)}{2} +n^2\right) + n^2|S_D\cup R|+ h-h_0 \\
    &=[F^+:\Q]\frac{n(n+1)}{2} +|T|n^2 +h-h_0 \\
    &=[F^+:\Q]n +|T|n^2 +h-[F^+:\Q]\frac{n(1-(-1)^{\mu_{\mathfrak{m}}-n})}{2}
\end{align*}

Consider the module $H^{\square}_{\infty}$. Since $H^{\square}_{\infty}$ is a finite free $S_{\infty}$ module, and since the action of ${S_{\infty}}$ factors through $R_{\infty}$, we see that 
\begin{equation*}
    \dim(S_{\infty})=\depth_{S_{\infty}}(H^{\square}_{\infty})\leq \depth_{R_{\infty}}(H^{\square}_{\infty})\leq \dim(R_{\infty})
\end{equation*}
and thus, the only possible way for this inequality to hold is if equality holds throughout. This implies $\mu_m\equiv n \mod 2$ and $H^{\square}_{\infty}$ is a maximal Cohen-Macaulay $R_{\infty}$ module. 

Now, consider the generic fibre. Let $m\subseteq R_{\infty}[1/l]$ be a maximal ideal. Lemma \ref{lem: Rloc} shows that $R_{\infty}[1/l]_m$ is a regular local ring.
Thus, any finitely generated maximal Cohen-Macaulay $R_{\infty}[1/l]_m$-module has finite projective dimension, and hence any maximal Cohen-Macaulay module is projective by the Auslander-Buchsbaum formula. 
This shows that $H^{\square}_{\infty}[1/l]_m$ is a free $R_{\infty}[1/l]_m$-module, this shows that $H^{\square}_{\infty}[1/l]$ is a locally finite free $R_{\infty}[1/l]$-module. It follows that $H_0[1/l]$ is a locally finite free $R^{univ}_{\mathcal{S}}[1/l]$-module. 
\end{proof}
\begin{corollary}
$R^{univ}_{\mathcal{S}}[1/l]=\T[1/l]$. 
\end{corollary}
\begin{proof}

Let $I$ be the kernel of the surjection $R^{univ}_{\mathcal{S}}[1/l]\rightarrow \T[1/l]$. Choose any maximal ideal $m$ of $R^{univ}_{\mathcal{S}}[1/l]$. Since localisation is an exact functor, we get a short exact sequence 
\[0\rightarrow I_m \rightarrow R^{univ}_{\mathcal{S}}[1/l]_m\rightarrow \T[1/l]_m\rightarrow 0.\]
Note that the action of $R^{univ}_{\mathcal{S}}[1/l]_m$ on $H_0[1/l]_m$ factors through $\T[1/l]_m$, so that $I_m$ annihilates all of $H_0[1/l]_m$. Since this is a free module, this shows that $I_m$ is trivial. Since this is true for every $m$, this shows that $\Supp(I)=\emptyset$ and hence $I=0$. 
Hence the surjection above is an isomorphism $R^{univ}_{\mathcal{S}}[1/l]\cong\T[1/l]$.

\end{proof}

\begin{remark}
    We finally want to remark on an application of Theorem \ref{theorem locfree}. Whenever $M$ is a locally free coherent sheaf on a connected space $X$, the rank function
    \begin{align*}
        X&\rightarrow \N\cup\{0\} \\
        x&\mapsto \textrm{Rank}_x(M)
    \end{align*}
    is locally constant. Therefore, the rank of a geometrically connected component can be calculated by calculating the rank at any special point $x\in X$. 
    In our special case, the rank of the module $H_0[1/l]$ can be interpreted as the number of distinct automorphic forms with a given set of Hecke eigenvalues, which can be interpreted as the multiplicity of the Galois representation determined by said Hecke eigenvalues inside the space of automorphic forms. 
    We have shown that for these automorphic forms, the multiplicity is determined only by the connected component of $R_{\infty}[1/l]$ on which the representation $\rho_{\mathfrak{m}}$ lies. By Lemma 4.2 of \cite{geraghty_2018}, we see that the minimal primes of $R_{\infty}[1/l]$ biject with the minimal primes of $\Lambda$.
    Thus, if one could show that for each component of $\Spec \Lambda$, there is an automorphic form of some classical weight had multiplicity 1, then all the Hida families of forms would also have multiplicity 1.     
\end{remark}

\bibliographystyle{alpha}
\bibliography{References}

\end{document}